\documentclass[reqno,oneside]{amsart}

\usepackage[foot]{amsaddr}

\usepackage{amssymb,amsmath,amsfonts,mathtools}

\usepackage{enumitem}

\usepackage{amsthm}

\newtheorem{theorem}{Theorem}[section]
\newtheorem{prop}[theorem]{Proposition}
\newtheorem{lemma}[theorem]{Lemma}

\theoremstyle{definition}
\newtheorem{definition}[theorem]{Definition}

\theoremstyle{remark}
\newtheorem{remark}[theorem]{Remark}
\newtheorem*{remark*}{Remark}

\numberwithin{equation}{section}

\newcommand{\ds}{\displaystyle}

\def\Rb{\mathbb R}

\def\Cc{\mathcal C}
\def\Dc{\mathcal D}

\def\veps{\varepsilon}
\def\vphi{\varphi}

\newcommand{\set}[2]{\left\{ #1 \mathrel{}\middle| \mathrel{} #2 \right\}}

\begin{document}

\title[Initial-boundary value problem for the Lifshitz-Slyozov equation]{The Initial-boundary value problem for the Lifshitz-Slyozov equation with non-smooth rates at the boundary}

\author[J.~Calvo]{Juan~Calvo$^{*,\lowercase{a}}$}
\address{$^*$Departamento de Matem\'atica Aplicada and Research Unit ``Modeling Nature'' (MNat), Universidad de Granada, Granada, Spain.}
\address{$^a$Corresponding author} \email{juancalvo@ugr.es} 

\author[E.~Hingant]{Erwan~Hingant$^{\dag,\S}$}
\address{$^\dag$Departamento de Matem\'atica, Universidad del B\'io-B\'io, Concepci\'on, Chile.}\email{ehingant@ubiobio.cl}

\author[R.~Yvinec]{Romain~Yvinec$^{\ddag,\$,\S}$}
\address{$^{\ddag}$PRC, INRAE, CNRS, Universit\'e de Tours, 37380 Nouzilly, France.}
\address{$^{\$}$Inria, Inria Saclay-Île-de-France, 91120 Palaiseau, France}
\address{$^{\S}$Associate member at Cogitamus Laboratory.}
\email{romain.yvinec@inrae.fr}

\date{\today}

\begin{abstract}
We prove existence and uniqueness of solutions to the initial-boundary value problem for the Lifshitz--Slyozov equation (a nonlinear transport equation on the half-line), focusing on the case of kinetic rates with unbounded derivative at the origin. Our theory covers in particular those cases with rates behaving as power laws at the origin, for which an inflow behavior is expected and a boundary condition describing nucleation phenomena needs to be imposed. The method we introduce here to prove existence is based on a formulation in terms of characteristics, with a careful analysis on the behavior near the singular boundary. As a byproduct we provide a general theory for linear continuity equations on a half-line with transport fields that degenerate at the boundary. We also address both the maximality and the uniqueness of inflow solutions to the Lifshitz--Slyozov model, exploiting monotonicity properties of the associated transport equation.  
\end{abstract}



\maketitle

\tableofcontents

\section{Introduction}

\subsection{The Lifshitz--Slyozov equation}

The purpose of this work is to provide a well-posedness theory for the Lifshitz--Slyozov model with inflow boundary conditions under widely general assumptions on the initial data and the kinetic rates. The Lifshitz--Slyozov system \cite{Lifshitz} describes the temporal evolution of a mixture of monomers and aggregates, where individual monomers can attach to or detach from already existing aggregates. The aggregate distribution follows a transport equation with respect to a size variable, whose transport rates are coupled to the dynamics of monomers through a mass conservation relation. The initial-boundary value problem for the Lifshitz--Slyozov model thus reads
\begin{equation} 
\label{eq:LS-N}
\left\{
 \begin{array}{l}
 \ds \frac{\partial f(t,x)}{\partial t} + \frac{\partial [(a(x)u(t)-b(x))f(t,x)]}{\partial x} = 0  \vphantom{\int} \, , \quad t>0\,, \ x\in(0,\infty)\,,\\[0.8em]
 \ds u(t) + \int_0^\infty xf(t,x)\, dx = \rho \, , \quad t>0\\[0.8em]
 \end{array}
 \right.
\end{equation}
for some given $\rho>0$, subject to the initial condition
\begin{equation}
\label{eq:LS-Nini}
f(0,x)=f^{in}(x)\,,\ \quad x\in (0,\infty)
\end{equation}
and the boundary condition 
\begin{equation}
\label{eq:LS-Nbc}
 \lim_{x\to 0^+} (a(x)u(t)-b(x))f(t,x) = \mathfrak n(u(t))\, , \quad t>0 
\end{equation} 
whenever \smash{$u(t)>\lim_{x\to 0^+} \tfrac{b(x)}{a(x)}$}. Here $f(t,x)$ is a nonnegative distribution of aggregates according to their size $x$ and time $t$, $u(t)$ is the monomer concentration and $\rho$ is interpreted as the total mass of the system. The kinetic rates $a(x)$ and $b(x)$ determine how fast do attachment (a given monomer attaches to a given aggregate) and detachment (a monomer detaches from a given aggregate) reactions take place. Aggregates change their size over time according to the quantity of monomers that they gain or lose through the previous reactions. Note that the attachment process is a second order kinetics, responsible of the nonlinearity, whereas detachment is a first order kinetics, as reflected in the transport term in \eqref{eq:LS-N}. The function $\mathfrak{n}$ in \eqref{eq:LS-Nbc} can be interpreted as a nucleation rate (i.e. the rate of formation of zero-size aggregates from monomers). At least formally, this rate governs the total number of aggregates as \smash{$\tfrac{d}{dt}\int_0^\infty f(t,x)dx=\mathfrak n(u(t))$} whenever $u(t)>\lim_{x\to 0^+} \tfrac{b(x)}{a(x)}$. This last condition means that the characteristic curves point towards the domain $x>0$ at time $t$, in which case a boundary condition must be specified, and is given by \eqref{eq:LS-Nbc}. Similar boundary conditions have been considered in \cite{Calvo,Collet2002,Deschamps}.

Writing the transport flow as $a(x)u(t)-b(x)=a(x)(u(t)-\Phi(x))$, with $\Phi:=b/a$, allows to appreciate the crucial role of the function $\Phi$. 
The latter measures the relative strength of detachment with respect to attachment, for a given aggregate size. 
Therefore, this single function includes most of the relevant information of the model. 
When $\Phi$ is monotonously decreasing, with $\lim_{x\to 0^+} \Phi(x)>\rho\geq u(t)$, large aggregates ($x>\Phi^{-1}(u(t))$) grow larger at the expense of smaller ones ($x<\Phi^{-1}(u(t))$), a phenomena called Ostwald ripening, and in such a case a boundary condition like \eqref{eq:LS-Nbc} is not needed. 
The Lifshitz--Slyozov model has been traditionally used to describe late stages of phase transitions, where the above mentioned Ostwald ripening phenomena take place: recall indeed that the classical Lifshitz--Slyozov rates are given by $a(x)=x^{1/3}$ and $b(x)=1$, see e.g. \cite{Niethammer2000}.  
In standard nucleation theory, a discrete size model analog, named the Becker-D\"oring model \cite{HY17}, is rather used to describe the initial stage of phase transition, where the nucleation process is the dominant one. 
Recently, the  intermediate stage has been considered in the physical literature \cite{Alexandrov13,Alexandrov20,Makoveeva,Shneidman10,Shneidman11}, where the growth of large aggregates and the  ongoing nucleation rate are of equal importance, leading to equations like \eqref{eq:LS-N}--\eqref{eq:LS-Nbc} or variants of it, with $\lim_{x\to 0^+} \Phi(x)< u^{in}=\rho-\int_0^\infty xf^{in}(x)$, in which case a boundary condition must be specified.
Indeed, some sets of kinetic rates for Eq.~\eqref{eq:LS-N} may lead to Ostwald ripening phenomena only after a certain transient period, where the dynamics of the Lifshitz--Slyozov model are driven by boundary effects at very small sizes, and for which the boundary term \eqref{eq:LS-Nbc} becomes important. 

Moreover, recent applications of this framework in biologically-oriented contexts utilize a different set of kinetic rates and then a boundary condition becomes mandatory in order to make sense of the model. 
A growing literature can be found on applications to protein polymerization phenomena and neurodegenerative diseases, starting from the so-called prion model and some of its variants (see e.g. \cite{Calvo,Doumic2009,Greer,Laurencot2007,Leis,Prigent,Simonett} and references therein), whose different versions come as modifications of the standard Lifshitz--Slyozov equations. 
Inflow boundary conditions are used to describe nucleation processes; the discrete models considered in \cite{Collet2002,Deschamps} are also related to this scenario by means of suitable scaling limits as we mention below. We also have in mind applications to modeling in Oceanography. 
For instance, the sea-surface microlayer (see e.g. \cite{Wurl}) is rich in conglomerates that grow in size by an aggregation process whereby particulate organic carbon  attaches to transparent exopolymeric particles; detachment effects can also take place and eventually additional terms may be included in \eqref{eq:LS-N}--\eqref{eq:LS-Nbc}, e.g. coagulation integrals. 
Tentative applications of variants of \eqref{eq:LS-N}--\eqref{eq:LS-Nbc} can be also envisioned where $x$ is a depth variable and the gradual sinking of aggregates (``marine snow'' \cite{Burd,Jackson}) proceeds by a ballasting process. We conjecture that more applications of this framework will gradually arise. 
The common feature is that the boundary condition \eqref{eq:LS-Nbc} can be interpreted as the synthesis of new aggregates from monomers and not necessarily by means of a mass action kinetics. 
The value $\Phi_0:=\lim_{x\to 0^+} \Phi(x)$ describes how strong are detachment effects compared to attachment effects for zero-size aggregates, which are precisely the ones formed by nucleation. Although the model does not account for the nucleation step in detail, nuclei are formed from monomers (by the function $\mathfrak n$). 
Unless the newborn aggregate is able to surmount a certain energetic barrier, it is unstable and dissolves immediately. Only stable aggregates persist long enough to grow larger by the addition of extra monomers. 
This stability issue is represented here by the value $\Phi_0$: the lower this value is, the more stable are these zero-sized seeds. With the boundary condition \eqref{eq:LS-Nbc}, we are representing an average behavior,  whereby nucleation is successful only when we have enough monomer availability, that is the condition $u(t)>\Phi_0$. 
See \cite{Deschamps} for more details on those lines, where the model \eqref{eq:LS-N}--\eqref{eq:LS-Nbc} is deduced as a scaling limit of the Becker--D\"oring model and the inflow boundary condition is interpreted in terms of the scaling and the mesoscopic reaction rates; note that some partial analysis in this direction were already given in \cite{Collet2002}. 

As far as we know, works covering mathematical aspects of the initial-boundary value problem for the Lifshitz--Slyozov model are presently scarce. We mention \cite{Calvo}, where it is shown that in some particular cases the model leads to dust formation (concentration at zero size), a behavior that can be somewhat prevented if fragmentation terms are incorporated into the model. Incidentally, the model with kinetic rates such that the boundary becomes characteristic is considered in \cite{Collet00}. Quite the contrary, the mathematical literature for the classical Lifshitz--Slyozov model is well established. Concerning density solutions, existence and uniqueness of mild solutions for Lipschitz rate functions is given in \cite{Collet00}, whereas existence and uniqueness of weak solutions for rates that need not be regular at the origin are provided in \cite{Laurencot01}. Measure solutions were considered in \cite{Collet00,Niethammer2000,Niethammer2005}. Mathematical justifications of the  connection between the Becker--D\"oring model and the Lifshitz--Slyozov model can be found in \cite{Collet2002,Laurencot2002,Niethammer03,Schlichting}; the results therein can be also understood as existence proofs. The long time behavior is analyzed in \cite{Collet00,Collet2002bis}, however our understanding of the dynamical behavior is not complete yet. Therefore, numerical simulations are a useful way to get further insights on the asymptotic behavior; some contributions along these lines are \cite{Carrillo,Goudon}. A number of variants of the Lifshitz--Slyozov model have been considered in the literature as well; we refer to \cite{Conlon,Goudon20,Hariz,Stoltz,Velazquez} for diffusive versions (also advocate to represent intermediate stages of aggregates growth) and to \cite{Laurencot2002LSW,Niethammer2000,Niethammer2005} for the Lifshitz--Slyozov--Wagner model.

In this contribution we study existence and uniqueness of local-in-time solutions for \eqref{eq:LS-N}--\eqref{eq:LS-Nbc}, together with continuation criteria and results on long-time behavior. To the best of our knowledge, this is the first contribution that tackles the well-posedness issue for the inflow boundary condition \eqref{eq:LS-Nbc}; therefore, our results cannot be directly compared with those given in classical works like \cite{Collet00,Laurencot01}. However, our methods of proof owe much to theirs, as we shall explain in the sequel. In order to tackle the well-posedness of \eqref{eq:LS-N}--\eqref{eq:LS-Nbc} we have chosen to use aa approach based on characteristics. This has the advantage of providing a semi-explicit representation formula (which may prove useful for e.g. designing particle methods) and is reminiscent of the works \cite{Collet00,Collet2002bis}. Due to the wide spectrum of  applications mentioned above, it is crucial to be able to cope with rates that are not regular at the origin. This generates a number of technical difficulties in order to make sense of characteristic curves, difficulties  that are not present when the rates are globally Lipschitz; one of the main contributions of this paper is to provide a reformulation that allows to give a suitable meaning to phase space trajectories/characteristic curves even when there is no forward-in-time uniqueness for those. We take definite advantage of working in  dimension one and represent solutions as a mixture of trajectories reaching the initial configuration or the boundary datum respectively, for every time instant. In such a way we are able to construct solutions unambiguously. Similar ideas belong to the folklore on boundary problems for transport equations, although we have not been able to find a suitable reference covering our non-Lipschitz regularity setting. Note in particular that we do not assume to have transport fields with bounded divergence (see \cite{Crippa14a,Crippa14b} for contributions in that direction); recall that the assumptions in \cite{diPerna} can be lifted in some cases, see e.g. \cite{Crippa,Desjardins}. Thus, for the reader's convenience we work out the full theory from scratch, which we believe to be of independent interest for the sake of other applications. Our construction guarantees that no singularities (shock formation, concentration phenomena) are created during the temporal evolution despite of the incoming boundary flow. We also extend the uniqueness proof in \cite{Laurencot01} to be able to cope with inflow solutions in this low-regularity context. As regards the scope of the theory we develop here, we provide examples of local solutions that can be extended to global ones and at the same time we clearly show why local-in-time existence of inflow solution is the best we can hope for generically. The breakdown of global existence is proved by giving examples of solutions that do not exist globally in time because the boundary condition loses its meaning, which raises the problem of giving a wider meaning to the solution concept in order to be able to extend every local solution to a global one. This is an important issue that is deeply connected with a full understanding of the long time behavior and will be tackled elsewhere by the authors and collaborators. 

We now state our main definitions and results in subsection \ref{ssec:main}, and give the general strategy of their proofs with the outline of the manuscript in subsection \ref{ssec:outline}.

\subsection{Definitions and main results}
\label{ssec:main}

Let us recall a few classical notations. Given a subset $\Omega$ of $\Rb^d$ equipped with the subspace topology, we denote by $\mathcal{C}^k(\Omega)$ the space of continuous real-valued function defined on $\Omega$ with at least $k$ continuous derivatives and $\mathcal{C}_c^k(\Omega)$ is its subspace consisting of compactly supported functions. When $\Omega$ is open, we also use $\mathcal D(\Omega)$, the space of infinitely differentiable real-valued functions defined on $\Omega$ with compact support, and $\mathcal D'(\Omega)$, its topological dual, the space of distributions on $\Omega$. For a measure $\mu$ defined on the borelian sets of $\Omega$ we understand by $L^1(\Omega,\mu)$, resp. $L^\infty(\Omega,\mu)$, the classical Lebesgue space consisting of the equivalence  class of $\mu$-integrable, resp. $\mu$-essentially bounded, real-valued functions defined on $\Omega$ agreeing $\mu$-almost everywhere (\emph{a.e.}). Recall that the weak topology on $L^1(\Omega,\mu)$, denoted by the prefix $w$, is the topology induced by the dual space $L^\infty(\Omega,\mu)$. The reference to the measure $\mu$ might be omitted if we clearly refer to the Lebesgue measure. The subscript \emph{loc} for \emph{locally} might be added to the Lebesgue space with the usual sense.  We will make use of two more spaces, for $X$ a Banach space and $I$ an interval: $\mathcal{C}(I,w-X)$ denotes the space of continuous $X$-valued functions defined on $I$, where $X$ is endowed with its weak topology and $L^\infty(I,X)$ is the Bochner space of essentially bounded $X$-valued functions defined on $I$ agreeing \emph{a.e.} with respect to the Lebesgue measure on $I$. Also, during the document we will use a notation like $C(A,B,\ldots)$ to denote a positive constant depending on the quantities between brackets, whose actual value is not relevant. Its value may change from line to line without explicit mention. Finally, for any $T\in(0,\infty]$ we set 
\[\Omega_T = [0,T)\times(0,\infty)\ \text{and}\ \Omega_T^* = (0,T)\times(0,\infty).\]
Note that when we refer to $T>0$ in what follows we shall always assume it is finite unless it is explicitly stated otherwise.

\medskip

In order to introduce our solution concept, let us give first the minimal regularity needed to define a solution to the problem \eqref{eq:LS-N}--\eqref{eq:LS-Nbc}.

\begin{definition}[Kinetic rates]
    A triplet $\{a,b,\mathfrak n\}$ defines kinetic rates provided that:
    \begin{enumerate} [align=left, leftmargin=2pt, labelindent=\parindent,label=\textup{(\arabic*)},itemsep=3pt]
    \item $a$ and $b$ are locally bounded and nonnegative functions on $[0,\infty)$,
    \item The function  $\Phi(x) \coloneqq b(x)/a(x)$ is defined for {\it a.e.} $x>0$ in $[0,\infty]$ and has a limit $\Phi_0\in[0,\infty]$ at $0^+$.
    \item $\mathfrak n$ is a locally bounded and nonnegative function on $[\Phi_0,\infty)$. 
    \end{enumerate}
\end{definition}

As detailed in introduction, the relative value of the monomer concentration $u$ with respect to the value $\Phi_0$ governs whether or not nucleation occurs. For any size $x\ge 0$ such that $\Phi(x)=u(t)$, the transport field vanishes and may change sign. We are interested in situations where the transport field points inwards in a neighborhood of zero; therefore, $\Phi_0$ plays a crucial role in our concept of solution.

\begin{definition}[Solution to the initial-boundary value problem] \label{def:inflow}
    Let $T\in(0,\infty]$. Assume to be given the kinetic rates $\{a,b,\mathfrak n\}$, a constant $\rho>0$ and a nonnegative function $f^{\rm in}$ belonging to $L^1(0,\infty)$. We say that a function $f$ is a solution to the Lifshitz--Slyozov equation on $[0,T)$ with  mass $\rho$, kinetic rates $\{a,b,\mathfrak n\}$ and initial value $f^{\rm in}$ if the following statements are satisfied:
    \begin{enumerate} [align=left, leftmargin=2pt, labelindent=\parindent,label=\textup{(\arabic*)},itemsep=3pt]
    \item The function $f$ belongs to $\mathcal C([0,T),w-L^1((0,\infty),x \,  dx))$, is nonnegative and for each $T^*<T$, it also belongs to  $L^\infty((0,T^*),L^1((0,\infty),dx))$;
    \item For all $t\in[0,T)$, 
    \begin{equation}\label{eq:def u}
    u(t)\coloneqq \rho-\int_0^\infty xf(t,x)\, dx >\Phi_0\,;
    \end{equation}
    \item For all $\vphi\in \mathcal{C}^1_c([0,T)\times[0,\infty))$, there holds that
    \begin{multline} \label{eq:LS-weak-time-density-boundary}
    \int_{0}^{T} \int_0^\infty \left(\partial_t \vphi(t,x) + (a(x)u(t)-b(x))\partial_x\vphi(t,x)\right) f(t,x)\, dx \, dt\\
    + \int_{0}^{T} \vphi(t,0)\mathfrak n(u(t)) \, dt  + \int_0^\infty \vphi(0,x)f^{\rm in}(x) \, dx = 0\,.
    \end{multline}
    \end{enumerate}
\end{definition}

To construct a solution to the Lifshitz--Slyozov equation we will assume that the kinetic rates $\{a,b,\mathfrak n\}$ satisfy the following working hypotheses:
\begin{align}
& a,\, b \in \mathcal C^0([0,\infty))\cap \mathcal C^1(0,\infty)\, ,  \tag{H1} \label{H1} \\
& a'\ \text{and} \  b' \ \text{are bounded on} \ (1,\infty)\,, \tag{H2} \label{H2} \\
& a(x)>0 \ \text{for all}\ x>0 \ \text{and} \ \tfrac 1 a \in L^1(0,1)\, , \tag{H3} \label{H3}
 \\
& \Phi' \in L^1(0,1)\, , \tag{H4} \label{H4} \\
& \mathfrak n \ \text{is continuous on} \ [\Phi_0,\infty)\,. \tag{H5} \label{H5}
\end{align}
Moreover, we restrict the choice of initial data to 
\begin{align}
& f^{\rm in} \in L^1((0,\infty),(1+x)\, dx)\,, \tag{H6} \label{H6}\\
& u^{\rm in} \coloneqq  \rho-\int_0^\infty xf^{\rm in}\, dx >\Phi_0\,, \tag{H7} \label{H7}
\end{align}
so that the balance of mass \eqref{eq:def u} makes sense at time $t=0$ together with the regularity required on $f^{\rm in}$ in Definition \ref{def:inflow}. Note that assumptions \eqref{H1} and \eqref{H2} are very similar to those considered in \cite{Laurencot01} and allow us to consider a larger set of kinetic rates than that of \cite{Collet00} as far as regularity is concerned. In particular, these ensure the existence of a positive constant $K_r$ such that
\begin{equation}\label{eq:sublinearity}
 a(x) + b(x) \leq K_r(1+x)
\end{equation}
for all $x\ge 0$. 
Recall that existence is not guaranteed for rates exhibiting strictly superlinear growth \cite{Collet00}.

Before we discuss further these assumptions and state our existence result, let us introduced a lemma that might help to interpret Definition \ref{def:inflow} through the standard moment equations, which in turn will be useful for several estimates in the sequel. 

\begin{lemma}[Moment equations] \label{lem:moments equations}
    Assume to be given the kinetic rates $\{a,b,\mathfrak n\}$ satisfying hypotheses \eqref{H1}-\eqref{H2} and \eqref{H5}, a constant $\rho>\Phi_0$ and a nonnegative function $f^{\rm in}$ satisfying \eqref{H6} {and \eqref{H7}}. Let $T>0$ and $f$ be a solution to the Lifshitz-Slyozov equation on $[0,T)$ with mass $\rho$, kinetic rates $\{a,b,\mathfrak n\}$ and initial value $f^{\rm in}$. For all $t\in[0,T)$ and for  every $h\in \mathcal C^0([0,\infty))$ such that $h'\in L^\infty(0,\infty)$, we have
    \begin{multline} \label{eq:LS-moment-equation}
    \int_0^\infty h(x)f(t,x)\, dx =  \int_0^\infty h(x)f^{\rm in}(x)\, dx  \\
    + \int_{0}^{t} \int_0^\infty (a(x)u(s)-b(x))h'(x)f(s,x)\, dx \, ds + \int_{0}^{t} h(0) \mathfrak n(u(s)) \, dt\,.
    \end{multline}
     Moreover, $f$ belongs to $L^\infty((0,T);L^1((0,\infty),(1+x)dx))$, and the derivative of $u$ on $(0,T)$ belongs to $L^\infty(0,T)$ and is given by
    \begin{equation} \label{eq:derivative of u}
    \frac{d u(t)}{dt} = - u(t) \int_0^\infty a(x)f(t,x)\, dx + \int_0^\infty b(x) f(t,x)\, dx\, ,
    \end{equation}
    for \emph{a.e.} $t\in(0,T)$.
\end{lemma}

\begin{proof}
First, plug $\vphi(t,x)=g(t)h(x)$ with $g\in\Cc^1_c((0,T))$ and $h\in\Cc^1_c([0,\infty))$ into Eq.~\eqref{eq:LS-weak-time-density-boundary} and observe that the distributional derivative of $\int_0^\infty h(x)f(t,x)\, dx$ belongs to  $L^\infty(0,T)$ by Eq.~\eqref{eq:def u}, \eqref{eq:sublinearity} and the regularity (point 1) in Definition \ref{def:inflow}. Then, the time continuity of $f$ yields $f(0,x)=f^{\rm in}(x)$ \emph{a.e.} $x>0$, so that  \eqref{eq:LS-moment-equation} holds for $h\in \mathcal C_c^1([0,\infty))$. Then a standard regularization procedure allows to consider $h$ continuous on $[0,\infty)$ with $h'\in L^\infty(0,\infty)$ in \eqref{eq:LS-moment-equation}, again, thanks to \eqref{eq:def u}, \eqref{eq:sublinearity} and the regularity (point 1) in Definition \ref{def:inflow}. The fact that $f$ belongs to $L^\infty((0,T);L^1((0,\infty),dx))$ follows by taking $h=1$ in \eqref{eq:LS-moment-equation} and the fact that $\mathfrak n$ is bounded on $[0,\rho]$ and $u(t)\in[\Phi_0,\rho]$ for all $t\in[0,T)$. Using  $h(x)=x$ and Gronwall's lemma yields that $f$ also belongs to $L^\infty((0,T);L^1((0,\infty),xdx))$. Finally, by the definition of $u$, we identify its derivatives. 
\end{proof}

\begin{theorem}[Existence of solution] \label{thm:existence}
    Assume to be given the kinetic rates $\{a,b,\mathfrak n\}$ satisfying hypotheses \eqref{H1} to \eqref{H5}, a constant $\rho>\Phi_0$ and a nonnegative function $f^{\rm in}$ satisfying \eqref{H6} and \eqref{H7}. There exists $T\in(0,+\infty]$ and a function $f$ belonging to $\mathcal C([0,T);w-L^1((0,\infty),(1+x)\, dx))$ satisfying:
    \begin{enumerate} [align=left, leftmargin=2pt, labelindent=\parindent,label=\textup{(\arabic*)},itemsep=3pt]
        \item $f$ is a solution to the Lifshitz--Slyozov equation on $[0,T)$ with  mass $\rho$, kinetic rates $\{a,b,\mathfrak n\}$ and initial value $f^{\rm in}$;
        \item $f(0,x)=f^{\rm in}(x)$ for a.e. $x>0$ and
    \[\lim_{x\to 0^+} (a(x)u(t)-b(x))f(t,x) = \mathfrak n(u(t))\,,\]
    for all $t\in(0,T)$;
        \item Either $T=\infty$ or $T<\infty$ and $\lim_{t\to T} u(t) = \Phi_0$.
    \end{enumerate}
\end{theorem}

\begin{remark} In the above theorem, besides proving the existence of a solution (point 1), we also prove in point 2 that the solution can be chosen with a trace at the origin and in point 3 we address its maximality. 
\end{remark}

\begin{remark}\label{rem:regulariy}
 We shall use  continuation arguments for $f$ several times in the sequel. It is worth noticing that, for $T$ finite, the time continuity holds on $[0,T]$ into $w-L^1((0,\infty),(1+x)dx)$. The regularity of the solution $f$ in this Theorem can be complemented as follows: (i) the sublinearity of the rates in \eqref{eq:sublinearity} entails that any moment in $x^\theta$ with $\theta\in[0,1]$ is time-continuous, 
 (ii) formula \eqref{eq:derivative of u} shows that $u$ is continuously differentiable. 
 \end{remark}

Hypotheses \eqref{H1} to \eqref{H4} fit well with power law rates: $a(x)=a_0x^{\alpha}$ and $b(x)=b_0x^{\beta}$ for $x\ge 0$ in the relevant case $0\leq \alpha \le \beta\le1$ with $a_0>0$, $b_0\ge 0$ and $\alpha<1$. Note that $\Phi(x)=\tfrac{b_0}{a_0} x^{\beta-\alpha}$ is such that $\Phi'$ is integrable at the origin and $\Phi_0=b_0/a_0$ if $\alpha=\beta$, while $\Phi_0 = 0$ if $\alpha<\beta$. The case $\alpha>\beta$ is out of the scope of this paper since $\Phi_0=\infty>\rho$ and then  the flow is outgoing. Hypothesis \eqref{H5} is trivially satisfied for $\mathfrak n(z)=\mathfrak n_0 z^n$ for $z\ge 0$ with $n\ge 1$ and $\mathfrak n_0\ge 0$, which is the typical situation we have in mind. Condition \eqref{H6} on initial data seems to be optimal to make sense of the mass balance for the initial datum and to be able to account for the boundary in the formulation \eqref{eq:LS-weak-time-density-boundary}. Finally, hypothesis \eqref{H7} is essential so that we may consider inflow solutions right from the initial time.

\begin{remark}
In this paper we work with rates $a$ and $b$ having classical regularity on $(0,\infty)$; this can be relaxed to Lipschitz regularity. The actual difficulty in the analysis comes rather from the lack of regularity at the origin (which includes the case of power law rates) combined with the boundary condition. In particular $a'$ and $b'$ need not be bounded around zero. The need for the integrability of $1/a$ is related to the method of factorization of the flow we consider here and works well for power laws too. Indeed, we rewrite the flow as $a(x)u(t)-b(x)=a(x)(u(t)-\Phi(x))$ and we consider (see Annex) a reparametrized flow of the form  $V(t,x)=u(t)-\Phi\circ A^{-1}(x)$, where $A$ is the primitive of $1/a$. If $1/a$ is not integrable around zero, the return time of the characteristic towards the boundary is infinite, in which case no boundary condition is needed. We also  mention that the integrability of $\Phi'$,  which  is equivalent to the integrability of $(\Phi\circ A^{-1})'$, is a standard assumption on the flow of a transport equation, namely $V\in L^\infty((0,T),W^{1,1}((0,R))\, \forall R>0$.  
\end{remark}

The solution constructed in Theorem \ref{thm:existence} can be shown to be unique under two additional assumptions. First, we strengthen \eqref{H5} by 
\begin{equation}\tag{H5'}\label{H5'}
 \mathfrak{n} \text{ is locally Lipschitz on } [\Phi_0,+\infty)\,,
\end{equation}
(which is satisfied e.g. for mass action kinetics) and we need some monotonicity of the function $\Phi$ around zero, namely
\begin{align}
& \text{There exists } x^*>0 \text{ such that } \Phi  \text{ is monotone on } [0,x^*). \tag{H8} \label{H8}
\end{align}

\begin{theorem}[Uniqueness of solution]\label{thm:uniqueness}
	Under the hypotheses of Theorem \ref{thm:existence}, assume moreover \eqref{H5'} and \eqref{H8} to be true. For all $T>0$, there exists at most one solution to the Lifshitz--Slyozov equation on $[0,T)$ with  mass $\rho$, kinetic rates $\{a,b,\mathfrak n\}$ and initial value $f^{\rm in}$.
\end{theorem}

Assumption \eqref{H8} is purely technical and avoids very irregular pathological situations like unbounded oscillations near the origin for the  function $\Phi$. It is clearly satisfied for $a$ and $b$ being power laws or other smooth functions and therefore not very restrictive in applications. Actually, we show in Section \ref{sec:uniqueness} below that Theorem \ref{thm:uniqueness} can be proved under slightly more general assumptions on the kinetic rates, see the assumptions (\ref{H8a})--(\ref{H8b}) in that section. 

We finish this section by a theorem giving sufficient conditions for global solution to exist, as well as providing examples of maximal solutions defined in a finite time interval.

\begin{theorem}[Global and local solutions] \label{thm:global local}
Let $f$ be a solution to the Lifshitz--Slyozov equation on $[0,T)$ with  mass $\rho$, kinetic rates $\{a,b,\mathfrak n\}$, initial value $f^{\rm in}$ and with $T<\infty$.
Under the hypotheses of Theorem \ref{thm:existence}, the following statements hold: 
	\begin{enumerate} [align=left, leftmargin=2pt, labelindent=\parindent,label=\textup{(\arabic*)},itemsep=3pt]
	\item 
	Assume $\Phi(x) \geq \Phi_0$ for all $x>0$. Then, for the prescribed rates and initial value there exists a global solution $f\in \mathcal C([0,\infty);w-L^1((0,\infty),(1+x)\, dx))$.
	\item Assume that $f^{\rm in}$ is compactly supported, that $\Phi$ is convex and strictly decreasing and that there exists numbers $\underline a$, $\overline a$ such that $0< \underline a < a(x) <\overline  a <\infty$ for all $x>0$. Then, there is no global solution $f\in \mathcal C([0,\infty);w-L^1((0,\infty),(1+x)\, dx))$ for the prescribed rates and initial value.

    \end{enumerate}
\end{theorem}

\begin{remark}
Point 1 covers the case of power law rates $b(x)=b_0x^{\beta}$ and $a(x)=a_0x^{\alpha}$ with $0\le \alpha < \beta < 1$. Note that when $\Phi_0=0$ we always have global existence. Point 2 states that any solution (regular enough) has to be local ($T<+\infty$); actually, the solution provided by Theorem \ref{thm:existence} in that case is in fact maximal because $\lim_{t\to T} u(t)=\Phi_0$.
\end{remark}

\begin{remark}
During the proof of point 2, we show that $u$ reaches $\Phi_0$ with a negative time derivative. This would allow to extend smoothly this solution past the time at which $u$ reaches $\Phi_0$ into an outflow solution for some time interval; this calls for a broader concept of solution to the Lifshitz-Slyozov equation, which would unify inflow and outflow solutions. Note that the situation is completely symmetric, in the sense that the arguments given in Section \ref{sec:global} can be adapted to  construct an outflow solution for which $u$ stays below $\Phi_0$ only for a finite time interval. 
\end{remark}

\subsection{Outline and methods of proofs}\label{ssec:outline}
We prove our existence result, Theorem \ref{thm:existence}, by means of a Schauder fixed point on $u$. This method was used before in \cite{Collet00} and makes use of representation formulas along characteristics in order to prove the continuity of the operator involved in the fixed point argument. Therefore, a detailed study of the {\em linear problem} (that is the continuity equation in \eqref{eq:LS-N} with known transport field \emph{i.e.} $u$ given in advance) together with \eqref{eq:LS-Nbc} is needed. This study can be performed in greater generality for a  broad class of degenerate transport fields that includes the one in \eqref{eq:LS-N}, an analysis that we deem of independent interest. Since this material is quite technical, we chose to quote the main results of this theory in Section \ref{sec:inflow}, and provide the details in an Annex. Once we have introduced the aforementioned machinery we can proceed to the analysis of the full nonlinear set of equations in Section \ref{sec:nonlinear}.  We start with the fixed point argument; this is done in Section \ref{sec:existence-inflow}. In fact, the fixed point strategy gives the existence of local-in-time solutions together with a continuation criterion: either  $T=\infty$ or $T<\infty$ with $u(t)\to \Phi_0$ as $t\to T$. Uniqueness of solutions, Theorem \ref{thm:uniqueness}, is proved in Section \ref{sec:uniqueness}; for that aim, we adapt the technique in \cite{Laurencot01}, which consists on proving Gronwall-type estimates for the tails densities. Finally, some examples of kinetic rates are discussed in Section \ref{sec:global} for which either global solutions can be constructed or local solutions cannot be extended further in time, which is Theorem \ref{thm:global local}. We complement the document with two annexes. In Section \ref{sec:annex}, we include a general framework to tackle a class of linear transport equations on a half-line with inflow boundary conditions and degenerate transport fields. We prove representation formulas along characteristic curves and integrability properties of the solutions thus given. Note that this annex   is written is such way that it can be read independently. We also include in a second annex, Section \ref{Ann:6}, some auxiliary results (somehow already contained in \cite{Laurencot01} without proofs) that are used in the uniqueness proof; their proofs are technically involved and placing them here allows for an easier navigation of the main text.

\section{Overview of the linear problem}

\label{sec:inflow}

All along this section we assume to be given $T>0$, $\rho>0$,  $\{a,b,\mathfrak{n}\}$ kinetic rates, and $u$ a function belonging to 
\[ \mathcal{BC}^+_\rho([0,T)) = \set{u}{ u :[0,T) \to [0,\rho]\ \text{continuous}} \,.\]
We denote 
\[ \overline u_T = \sup \set{u(t)}{ t\in[0,T)} \,, \quad \text{and} \quad \underline u_T = \inf \set{u(t)}{t\in[0,T)}\, .\]
Remark that, by definition, $0\leq \underline u_T \le \overline u_T \le \rho$.  Moreover, we assume that  $\underline u_T>\Phi_0$ and also that assumptions \eqref{H1} to \eqref{H5} hold. We define, for all $(t,x)\in\Omega_T$,
\[v(t,x):=a(x)u(t)-b(x)\:.\] 
During the rest of this Section we present several statements and properties that will be crucial for the existence proof in Section \ref{sec:nonlinear}. All of them will be proved, in greater generality, in the annex in Section \ref{sec:annex}.

\begin{lemma} 
    For any $(t,x)\in\Omega_T$, there exists a unique maximal solution to 
    \begin{equation} \label{eq:characteristics curves_maintext}
        \begin{array}{l}
        \displaystyle \dfrac{\partial }{\partial s}X(s;t,x) = v(s,X(s;t,x)) \, , \\[0.8em]
        \displaystyle X(t;t,x)=x 
    \end{array}
    \end{equation}
whose maximal interval is denoted by $\Sigma_{t,x}$. For every $s\in \Sigma_{t,x}$ we have  
\begin{equation} \label{eq:characteristics curves - derivatives in x and t - sec2}
\ds \dfrac{\partial }{\partial x}  X(s;t,x) \coloneqq J(s;t,x) = \exp \left( - \int_s^t  \left(\frac{\partial v}{\partial x} \right) (\tau,X(\tau;t,x))\, d\tau \right)\:.
\end{equation}
Moreover, as a consequence of \eqref{eq:sublinearity}, there exists a positive constant $C(T)$, independent of $u\in\mathcal{BC}^+_\rho([0,T))$, such that, for all $(t,x)$ in $\Omega_T$ and $s$ in $\Sigma_{t,x}$, we have
  \begin{equation} \label{eq:characteristics curves - uniform bounds_maintext}
	    X(s;t,x) + \left| \dfrac{\partial }{\partial s}X(s;t,x) \right|  \leq C(T)(1+x)\:. 
    \end{equation}
\end{lemma}

In order to construct a solution to the Lifshitz--Slyozov equation \eqref{eq:LS-N} through the so-called characteristics formulation, we need to know the lifetime of these characteristics, given by the lower and upper ends of $\Sigma_{t,x}$. Particularly, we need to identify which characteristics go back to some positive $x$ at time $s=0$ and which ones go back to the boundary $x=0$ in positive time $s>0$. We can translate this problem into the study of the time
\[ \sigma_{t}(x) \coloneqq \inf \{ s>0\,, s\in \Sigma_{t,x}\} \]
for each $(t,x)\in\, \Omega_T$,  which represents the time the curve $s\mapsto (s,X(s;t,x))$ reaches one of the two boundary axes $x=0$ or $s=0$ in $\Omega_T$. Note that for $t=0$, we readily have $\sigma_{0}(x)=0$ for every $x>0$.

Next we provide a rigorous sense for the concept of characteristic curves starting from $x=0$ at a positive time; this cannot be achieved directly from 
\eqref{eq:characteristics curves_maintext} due to the lack of derivative at the origin. Nevertheless, the analysis of the map $x\mapsto \sigma_t(x)$ at each time $t$ allows us to single out a unique characteristic curve starting from $x=0$ at a time $s\in(0,t)$. This provides an interpretation of $X(t;s,0)$ as the inverse of $\sigma_t(x)$, that is  $X(t;s,0)=\sigma_t^{-1}(s)$. These considerations are intimately related with the fact that $a(x)$ is the driving term at $x=0$ in the differential equation~\eqref{eq:characteristics curves} whenever $\underline u_T>\Phi_0$. Namely,
\[ v(t,x)=a(x) (u(t)-\Phi(x))= a(x) [u(t)-\Phi_0 + (\Phi_0-\Phi(x))] \]
but $\Phi_0-\Phi(x)$ has little influence when $x$ is close to the origin. Then an integrability condition for $1/a$ at the origin, by assumption \eqref{H3}, arises naturally, see e.g. \cite{Crippa}.

We shall show in the Annex (\emph{c.f.} Lemma \ref{lem:either-or property for sigma}) that when $\sigma_t(x)>0$, the characteristic curve reaches the axis $x=0$ at time $\sigma_t(x)$. Moreover, uniqueness of solutions to~\eqref{eq:characteristics curves_maintext} yields that the family of characteristic curves is a totally ordered family; therefore, we may tell whether characteristic curves came back from zero or not in terms of the separating point 
\begin{equation*}
    x_c(t)\coloneqq\inf\set{x>0}{\sigma_t(x)=0} 
\end{equation*}
defined for each $t$ in $[0,T)$. It can be proved that this defines a positive number, such that $\sigma_t$ is positive and nonincreasing in $(0,x_c(t))$.
In fact $t\mapsto x_c(t)$ can be interpreted as the characteristic curve starting from $x=0$ at time zero, as we state below. Also, note that the characteristic curves $s\mapsto X(s;t,x)$ do not leave $\Omega_T$ for $s\in (t,T)$, justifying the terminology of ``inflow''. Now we state a result that paves the way for the use of characteristics. 

\begin{prop} \label{prop:diffeomorphism_merge}
    For each $t\in(0,T)$, the map $x\mapsto X(t;0,x)$ is an increasing $\mathcal{C}^1$-diffeo\-morph\-ism from $(0,\infty)$ to $(x_c(t),\infty)$ with derivative given by $J(t;0,x)$ in Eq.~\eqref{eq:characteristics curves - derivatives in x and t - sec2} and the map $s\mapsto\sigma_t^{-1}(s)$ is a decreasing $\mathcal{C}^1$-diffeomorphism from $(0,t)$ to $(0,x_c(t))$ satisfying, for some constant $C(T)$ independent of the given $u \in \mathcal{BC}^+_\rho$ and $t\in(0,T)$, that $\sigma_t^{-1}(s) \leq C(T)$ for all $s\in (0,t)$. Moreover, we have that $\lim_{x\to 0^+}X(t;0,x)=x_c(t)$ and $\sigma_t^{-1}(s) = \lim_{x\to 0^+} X(t;s,x)$ for all $(t,x)\in\Omega_T^*$.
\end{prop}

 Once we have these statements we can provide a representation formula for the solutions of the linear problem. Let $f^{\rm in}$ be a nonnegative measurable function on $(0,\infty)$. Thanks to Proposition \ref{prop:diffeomorphism_merge}, we define for a.e. $(t,x)\in \Omega_T^*$ 
\begin{equation}
\label{eq:mild solution_maintext}
f(t,x) \!= \!f^{\rm in}(X(0;t,x)) J(0;t,x) \mathbf 1_{(x_c(t),\infty)}(x) \!+\! \mathfrak n(u(\sigma_t(x)))|\sigma_t'(x)|\mathbf 1_{(0,x_c(t))}(x)\,,
\end{equation}
where $\mathbf 1_I$ stands for the indicator function of an interval $I$. Indeed, $\sigma_t'(x) = 1/(\sigma_t^{-1}{}'(\sigma_t(x))$ is defined for all $t\in(0,T)$ and $x\in(0,x_c(t))$. Note that \eqref{eq:mild solution_maintext} makes sense as null sets are mapped to null sets under the diffeomorphism in Proposition \ref{prop:diffeomorphism_merge}. In view of the results obtained in the annex in Section \ref{sec:annex}, this construction satisfies the next proposition. 
\begin{prop} \label{eq:f before schauder}
 Assume $f^{\rm in}$ satisfies \eqref{H6} and $u(t)$ is given in advance. Then $f$ given by  \eqref{eq:mild solution_maintext} belongs to the space $\mathcal C\left([0,T);w-L^1((0,\infty),(1+x)dx)\right)$, satisfies the weak formulation \eqref{eq:LS-weak-time-density-boundary} and it also satisfies point 2 of Theorem \ref{thm:existence}.  Similarly, $f$ satisfies the moment equation \eqref{eq:LS-moment-equation} too.
\end{prop}
In the light of this result, to prove Theorem \ref{thm:existence} we are to couple \eqref{eq:mild solution_maintext} with \eqref{eq:def u}. This is what we do in the next section.

\section{The nonlinear problem}
\label{sec:nonlinear}

\subsection{Existence of solutions} \label{sec:existence-inflow}

We follow the lines of \cite{Collet00} to show existence of local-in-time inflow solutions via the Schauder fixed point theorem. All along this section we assume to be given $T>0$, $\rho>0$ and $\{a,b,\mathfrak{n}\}$ kinetic rates. Moreover, we assume that $ \Phi_0 < \rho$ and we let $f^{\rm in}\in L^1((0,\infty),(1+x)\,dx)$ be nonnegative and such that
\[ u^{\rm in} \coloneqq \rho - \int_0^\infty xf^{\rm in}(x)\,dx > \Phi_0 \,.\]
Let $\delta>0$ such that $2\delta < u^{\rm in} -\Phi_0$, and define 
\[ \mathcal B_\delta ([0,T))= \set{ u\in  \mathcal{BC}_\rho^+([0,T))}{ u(0)=u^{\rm in} \ \text{and}\ \Phi_0+\delta \leq u(t) \leq \rho, \ \forall t\in[0,T)}\,.\]
For each $u\in\mathcal B_\delta([0,T))$, we can define the density $f$  given by Eq.~\eqref{eq:mild solution_maintext} and then the function
\[ v(t)  = G(u)(t) = \left( \rho- \int_0^\infty xf(t,x)\, dx \right) \vee (\Phi_0+\delta)\]
for all $t\in[0,T)$ where $x\vee y$ denotes the maximum between $x$ and $y$ in $\mathbb R$. Our aim in this section is to prove the existence of a fixed point for the operator $u\mapsto G(u)$. We observe by construction that $\Phi_0+\delta\leq v(t) \leq \rho$; moreover, a straightforward consequence of Theorem \ref{th:compilation}
in the Annex is that the first moment of $f$ has a derivative belonging to $L^\infty(0,T)$ thanks to \eqref{eq:sublinearity}. So has $v$, and we identify it, for \emph{a.e.} $t\in(0,T)$, as
\begin{equation}
\nonumber
v'(t)  = \begin{cases}
\displaystyle - \frac{d}{dt} \int_0^\infty xf(t,x)\,dx, & \text{if}\ \rho - \int_0^\infty xf(t,x)\,dx \geq \Phi_0 + \delta \\
0, & \text{otherwise,}
\end{cases}
\end{equation}
where 
\[\frac{d}{dt}\int_0^\infty x f(t,x)dx = \int_0^\infty (a(x)u(t)-b(x))f(t,x)dx.\]
Hence $v$ is continuous, and thus $G$ is a map from $\mathcal B_\delta([0,T))$ into itself. Moreover, it follows from Lemma \ref{lem:compactness} in the Annex that the derivative $v'$ above is uniformly bounded on $(0,T)$, independently on $u$. Thus, the image of $\mathcal B_\delta([0,T))$ is compact for the uniform topology. The remainder of this section is devoted to prove the continuity of the operator $G$ and then Theorem \ref{thm:existence}.

In the sequel, for a given sequence $\{u^n\}$ in $\mathcal{B_\delta}$, we denote by $X^n$ the solution to Eq.~\eqref{eq:characteristics curves_maintext} associated with $u^n$ and we denote by $\sigma_t^{-1,n}$ the inverse function of $\sigma_t^n$ associated with $X^n$.      

\begin{lemma} \label{lem:limit-Xn}
	Let $\{u^n\}$ be a sequence in $\mathcal B_\delta([0,T))$ converging (uniformly) to $u$.  For each $x>0$,  $\{X^n(\cdot ;0,x)\}$ converges uniformly to $X(\cdot;0,x)$ on $[0,T)$ as $n\to \infty$.
\end{lemma}

\begin{proof}
Fix $x>0$. Thanks to the bounds in Eq.~\eqref{eq:characteristics curves - uniform bounds_maintext} and the continuity in the second variable of $X^n$, 
\[ |X^n(s;0,x)| \leq C(T)(1+x) \]
for all $s\in(0,T)$, with some constant $C(T)>0$ independent on  $n$. Moreover,
\[ \left|{\frac{\partial}{\partial s}} X^n(s;0,x)\right| \leq K (\rho +1) T \]
where $K$ can be taken as the maximum of $a$ and $b$ on the interval $[0,C(T)(1+x)]$. Thus the sequence $\{X^n(\cdot;0,x)\}$ is relatively compact and, up to a subsequence, converges to a continuous function $Y$ on $[0,T)$. Inspecting the equation on $X^n$ we realize that the limit satisfies
\[ Y(s) = x + \int_0^s (a(Y(\tau))u(\tau) -b(Y(\tau)))\,  d\tau\,.\]
Thus $Y$ is the unique solution to Eq.~\eqref{eq:characteristics curves_maintext} with $Y(0)=x$ and therefore it coincides with the characteristic curve $s\mapsto X(s;0,x)$ associated to $u$. By uniqueness of the limit the full sequence converges and the result follows. 
\end{proof}

\begin{lemma} \label{lem:limit-sigman}
	Let $\{u^n\}$ be a sequence in $\mathcal B_\delta([0,T))$ converging (uniformly) to $u$.  For each $t\in(0,T)$, $\{\sigma_t^{-1,n}\}$ converges pointwise to $\sigma_t^{-1}$ as $n\to \infty$.
\end{lemma}

\begin{proof}
Let $t\in(0,T)$ and $s\in(0,t)$. Define $x^n =\sigma_t^{-1,n}(s)$ for each $n\geq 1$. By Proposition \ref{prop:diffeomorphism_merge} the sequence $\{x^n\}$ is bounded; denote this bound by $\bar x$. Consider a subsequence of $\{x^n\}$ (not relabelled) which converges to some  $x$. Thanks to  Eq.~\eqref{eq:characteristics curves - uniform bounds_maintext} there is  a constant $C(T)$ such that for all $\tau \in(s,t)$ and $n\geq 1$,
\[ X^n(\tau;t,x^n) \leq x^* \coloneqq C(T)(1+\bar x)\,. \]
Then, by Eq.~\eqref{eq:characteristics curves_maintext},
\[ \left| {\frac{\partial}{\partial \tau}}X^n(\tau;t,x^n)\right| \leq \sup_{x\in(0,x^*)} ( |a(x)|\rho + |b(x)| ).\]
Hence, up to a subsequence, the sequence of functions $\tau \mapsto X^n(\tau;t,x^n)$ converges uniformly on $[s,t]$ to a continuous function, which we denote by $\tau \mapsto Y(\tau)$. Moreover, for all $n\geq 1$ and $\tau \in [s,t]$ we have
\[x^n - X^n(\tau;t,x^n) = \int_\tau^t [a(X^n(r;t,x^n))u^n(r) - b(X^n(r;t;x^n))]\, dr, \]
and at the limit $n\to \infty$,
\[ Y(\tau) = x - \int_\tau^t  [a(Y(r))u(r) - b(Y(\tau))]\,dr\,. \]
We observe that $Y$ solves Eq.~\eqref{eq:characteristics curves_maintext} with initial data $Y(t)=x$ on $[s,t]$, so $Y(\tau) = X(\tau;t, x)$ by uniqueness and in particular $\sigma_t(x)\leq s$. Finally, since $X^n(s;t,x^n)=0$ for all $n\geq 1$, we have at the limit that $Y(s)=X(s;t,x) = 0$ and so $\sigma_t(x)=s$. In conclusion, from any subsequence of $\{\sigma_t^{-1,n}(s)\}$ we can extract a subsequence converging to $\sigma_t^{-1}(s)$, so the full sequence converges. 
\end{proof}

We are now ready to prove the continuity of $G$.

\begin{prop}
	Let $T>0$. The operator $G$ is continuous on ${\mathcal{B}_\delta([0,T))}$.
\end{prop}

\begin{proof}
Let $\{u^n\}$ be a sequence in ${\mathcal{B}_\delta([0,T))}$ converging uniformly to $u$. Let $f^n$ be the function associated with $u^n$ that is given by Eq.~\eqref{eq:mild solution_maintext}. Thus
\[\int_0^\infty x f^n(t,x)\,dx = \int_0^\infty X^n(t,0,x) f^{\rm in}(x)\,dx + \int_0^t \sigma^{n,-1}_t(s)\mathfrak n (u^n(s))\,ds\,,\]
for all $t\in(0,T)$. Combining Lemmas \ref{lem:limit-Xn} and \ref{lem:limit-sigman} with the bounds in Eq.~\eqref{eq:characteristics curves - uniform bounds_maintext} and Proposition \ref{prop:diffeomorphism_merge}, we can use the dominated convergence theorem to show that
\[ \int_0^\infty xf^n(t,x)\, dx \to \int_0^\infty xf(t,x)\,dx\]
for all $t\in(0,T)$, where $f$ is the function associated with $u$ that is given by Eq.~\eqref{eq:mild solution_maintext}.  Thus $v^n(t)=G(u^n)(t)$ converges to $v(t)=G(u)(t)$ for all $t\in(0,T)$. Since the derivatives of $v^n$ are uniformly bounded in $L^\infty (0,T)$, as mentioned above, the convergence is uniform. 
\end{proof}

\begin{proof}[Proof of Theorem \ref{thm:existence}] The previous developments in this section enable us to apply Schauder's fixed point theorem. Thus, there exists a fixed point to $G$, which means that there is some $u \in \mathcal{B}_\delta([0,T))$ such that
\[ u(t) = \left( \rho- \int_0^\infty xf(t,x)\,dx \right) \vee (\Phi_0+{\delta})\]
for all $t\in[0,T)$, where $f$ is given by Eq.~\eqref{eq:mild solution_maintext} in terms of  $u$. Recall that $u(0)=u^{\rm in} >\Phi_0+2\delta$; thus, there exists $t^*$ such that $u(t) \geq \Phi_0+{\delta}$ for all $t\in[0,t^*]$ and hence for all $t\in[0,t^*]$,
\[u(t) = \rho- \int_0^\infty xf(t,x)\,dx\,.\]
This provides a solution on $[0,t^*]$ thanks to the considerations in Prop. \ref{eq:f before schauder}. Repeating this procedure we can  construct an increasing sequence of times $\{t_n\}$ such that we have a solution $f$ to our problem up to time $t_n$. We address now the maximality of this construction.  Assume that the limit of $\{t_n\}$ is finite and let us denote it by $T$. We show now that $\lim_{t\to T^-}u(t)= \Phi_0$  by a contradiction argument. Let us assume that $u$ does not converge to $\Phi_0$ at $T$. We would have a solution $f$ to the problem on $[0,T)$; however, since $u'$ is bounded on $(0,T)$, $u$ would have a limit at $T^-$. We would clearly have  $\lim_{t\to T^-} u(t) >\Phi_0$. This allows us to extend $f$ by continuity in $T$ (see Remark \ref{rem:regulariy}) and thus $f(T,\cdot)$ would belong to $L^1((0,\infty),(1+x)dx)$. In that case we can apply the fixed point procedure once more to obtain a solution on $[T,T+t^*)$ for some $t^*>0$, which contradicts the construction of the sequence $\{t_n\}$. Therefore, either $T=\infty$ or, $T<+\infty$ and $\lim_{t\to T^-}u(t)=\Phi_0$.  This ends the proof of Theorem \ref{thm:existence}. 
\end{proof}

\subsection{Uniqueness}\label{sec:uniqueness}

	The proof of Theorem \ref{thm:uniqueness} is based on a contraction strategy thanks to a Gronwall-type argument. The main idea, already used in \cite{Laurencot01} (and in \cite{Laurencot2002} on related Becker-D\"oring equations), is to work on the tail density rather than the density itself. The tail density solves a transport equation having convenient properties, such as a maximum principle.
	
	We consider $T>0$, $\rho>0$, $\{a,b,\mathfrak{n}\}$ kinetic rates and two nonnegative functions $f^{\rm in}_1$ and $ f^{\rm in}_2$ in $L^1((0,\infty),(1+x)\, dx)$. We consider two solutions $f_1$ and $f_2$ to the Lifshitz--Slyozov equation on $[0,T)$ with mass $\rho$, kinetic rates $\{a,b,\mathfrak n\}$ and   initial values $f^{\rm in}_1$ and $f^{\rm in}_2$ respectively. Let $u_1$ and $u_2$ be given by the mass conservation \eqref{eq:def u} respectively with the solutions $f_1$ and $f_2$, let also $v_1=au_1-b$ and $v_2=au_2-b$. We shall define the following tail density
	\begin{equation}\label{eq:definition of F}
	F_i(t,x)= \int_x^\infty f_i(t,y) \, dy\,,
	\end{equation}
	for $i=1,2$ and all $(t,x)\in\Omega_T$, being a continuous function, also we define
	\[E \coloneqq F_1-F_2\ \text{and}\ w = u_1 -  u_2\,.\]
	The following lemma is directly adapted from  \cite[lemma 5.1]{Laurencot01}.
	\begin{lemma} \label{lem:newlem_E}
	 Let $0\le \varphi \in \mathcal C^0([0,\infty))$ vanishing in a neighborhood of zero and such that $\vphi'\in L^\infty(0,\infty)$ is compactly supported. We have, for any $t\in[0,T]$,
		\begin{multline}\label{eq:pregronwall-Ep-regular}
			\int_0^\infty \vphi(x)|E(t,x)|\, dx  \leq   \int_0^\infty \vphi(x)|E(0,x)|\, dx \\
			+ \int_0^t \int_0^\infty \partial_x[v_1(s,x)\vphi(x)]  |E(s,x)| \, dx\, dt \\
			+ \int_0^t |w(s)|  \int_0^\infty a(x) \vphi(x) f_2(s,x) \, dx \, dt \,.
			\end{multline}
			Moreover, for any $t\in(0,T)$,
			\begin{equation}\label{eq:inter-w-1}
			|w(t)| \leq \int_0^\infty |E(t,x)|\, dx\,,
			\end{equation}
			and
			\begin{equation}\label{eq:control_on_E+(0)}
			|E(t,0)| \leq |E(0,0)| + K_{\mathfrak n} \int_0^t |w(s)| \, ds\,
			\end{equation}
			where $K_{\mathfrak n}$ is the Lipschitz constant of $\mathfrak n$ on $[\Phi_0,\rho]$.
		\end{lemma}
		\begin{proof}
	The proof of Lemma \ref{lem:newlem_E} is given in the annex in Section \ref{Ann:6} for the reader's convenience. 
\end{proof}

The idea now is to choose an admissible lower-bounded function $\varphi$ that satisfies
\[ \partial_x[v_1(s,x)\vphi(x)] \leq K\vphi(x)\,,\]
and
\[a(x)\varphi(x)\leq K(1+x)\,,\]
for some $K>0$, in order to combine Eqs.~\eqref{eq:pregronwall-Ep-regular}-\eqref{eq:inter-w-1} with a Gronwall argument to prove that $E$ and $w$ must be equal to zero whenever $E(0,x)=0$. Due to hypothesis \eqref{H2}, the only real difficulty is near the origin. It turns out that we can obtain such a test function $\varphi$ provided that $a\Phi'$ does not have unbounded oscillations around zero.
We thus distinguish two (non mutually exclusive) alternatives: 
\begin{align}
    & \exists C>0, x^*>0 \text{ such that } \forall x\in(0,x^*), -\Phi'(x)<\frac{C}{a(x)} \tag{H8a} \label{H8a}\,, \\
    & \exists C>0, x^*>0 \text{ such that } \forall x\in(0,x^*), -\Phi'(x)>\frac{C}{a(x)} \tag{H8b} \label{H8b}\,,
\end{align}
It is clear that assumption \eqref{H8} implies that at least one of the two cases \eqref{H8a} or \eqref{H8b} holds true. Conversely, \eqref{H8a} and \eqref{H8b} together allow for a more general set of kinetic rates than \eqref{H8} alone does. We are going to show in the sequel that any of these two hypotheses guarantees uniqueness. The next lemma provides explicitly the appropriate test function $\vphi$.

\begin{lemma}\label{lem:vphi} Let $\vphi$ be defined as follows:
    \begin{enumerate} [align=left, leftmargin=2pt, labelindent=\parindent,label=\textup{(\arabic*)},itemsep=3pt]
        \item If assumption \eqref{H8a} is true, we define
			\begin{equation}\label{def_vhpi_8a}
			\vphi(x)=
			\begin{cases}
			\frac{1}{a(x)}\,, &  x\leq \bar{x}\\
			\frac{1}{a(\bar{x})}\,, & x> \bar{x}
			\end{cases}
			\end{equation}
			for some given $\bar{x}$ (to be chosen later).
			\item If assumption \eqref{H8b} holds true, we define
			\begin{equation}\label{def_vhpi_8b}
			\vphi(x)=
			\begin{cases}
			\frac{1}{a(x)}\exp\left(-\int_x^{\bar{x}}\frac{C/a(y)+\Phi'(y)}{\delta}\, dy\right)\,, & x\leq \bar{x}\\
			\frac{1}{a(\bar{x})}\,, & x> \bar{x}%.
			\end{cases}
			\end{equation}
			for some given $\bar{x}$, $C$, $\delta$ \textup{(}to be chosen later\textup{)}.
		\end{enumerate}
		In both cases, with $\vphi$ defined either in \eqref{def_vhpi_8a} or \eqref{def_vhpi_8b}, we may choose the constant $\bar{x}$ \textup{(}and $C$, $\delta$ in the second case\textup{)} in a way  that there exists a constant $K>0$ such that, for all $x>0$ and all $t\in(0,T)$,
		\begin{equation} \label{eq:gronwall condition}
		\partial_x[v_1(t,x)\vphi(x)] \le K \vphi(x).
		\end{equation}
		Moreover, $\vphi$ is continuous on $(0,\infty)$ and continuously differentiable for all $x>0$ except at $\bar{x}$. It is bounded from below by
		\begin{equation}\label{eq:lowerbound_phi}
		\vphi(x)\geq 1/\|a\|_{L^\infty(0,\bar{x})}\,,
		\end{equation}
		and $a\vphi$ is bounded from above on $(0,\bar{x})$ by
		\begin{equation}
		\nonumber
		a(x)\vphi(x)\leq \max\left\{1,\exp\left(-\int_0^{\bar{x}}\frac{C/a(y)+\Phi'(y)}{\delta}\, dy\right)\right\}\,, \quad x<\bar{x}\,.
		\end{equation}
	\end{lemma}
	
	\begin{proof}
		Note that finding a constant $K>0$ such that Eq.~\eqref{eq:gronwall condition} holds is equivalent to  finding a constant $C>0$ such that
		\[(u_1(t)-\Phi(x))\partial_x (a\vphi)(x) \leq (C+a\Phi')\vphi(x)\,.\]
		Let us check that this inequality holds true for the function $\vphi$ defined in \eqref{def_vhpi_8a} or \eqref{def_vhpi_8b} and well chosen constants.
		
		We first deal with case \textup{(}1\textup{)}. Let $C$ and $x^*$ be defined from assumption \eqref{H8a}. For any $0<\bar{x}\leq x^*$, and for all $x\leq \bar{x}$, the function $\vphi$ defined in \eqref{def_vhpi_8a} satisfies 
		\[(u_1(t)-\Phi(x))\partial_x (a\vphi)(x)=0 \leq \left(C'+a(x)\Phi'(x)\right)\vphi(x)\,,\]
		for any $C'>C$, due to \eqref{H8a} and the fact that $\vphi$ is positive. For $x>\bar{x}$,
		\[\partial_x[v_1(t,x)\vphi(x)] \leq \vphi(\bar{x}) \left(\| a'(x)\|_{L^\infty(\bar{x},\infty)}\rho + \| b'(x)\|_{L^\infty(\bar{x},\infty)}\right) \leq C'' \vphi(x)\,,\]
		for any constant $C''\geq \left(\| a'(x)\|_{L^\infty(\bar{x},\infty)}\rho + \| b'(x)\|_{L^\infty(\bar{x},\infty)}\right)$. Thus Eq.~\eqref{eq:gronwall condition} holds true for any $x$ and for a sufficiently large constant $K$.
		
		Now let us deal with case \textup{(}2\textup{)}. Let $C$ and $x^*$ be defined from assumption \eqref{H8b}. Thanks to the continuity of $u_1(t)>\Phi_0$ and that of $\Phi$, we can show that there exists $\delta>0$ and $x_0>0$ such that \[\inf_{t\in (0,T)} u_1(t)> \sup_{x\in(0,x_0)}\Phi(x)+\delta\,.\]
		Let then $\bar{x}=\min(x^*,x_0)$. For $x\leq \bar{x}$, then $\vphi$ satisfies 
		\begin{multline*}
		(u_1(t)-\Phi(x))\partial_x (a\vphi)(x) = (u_1(t)-\Phi(x))\frac{C/a(x)+\Phi'(x)}{\delta}(a\vphi)(x) \\ \leq \left(C+a(x)\Phi'(x)\right)\vphi(x)\leq \left(C'+a(x)\Phi'(x)\right)\vphi(x)\,,
		\end{multline*}
		for any $C'\geq C$, as $u_1(t)-\Phi(x)\geq \delta$ but  $C/a(x)+\Phi'(x)<0$ and $\vphi$ is positive. The case $x> \bar{x}$ is managed as in the case \textup{(}1\textup{)} above.
	\end{proof}
	
	Note that as the function $\vphi$ is bounded from below by \eqref{eq:lowerbound_phi},  we have for some constant $C>0$,
	\begin{equation} \label{eq:bound E+}
	\int_0^{\infty} | E(t,x)|\, dx \leq C \int_0^\infty \vphi(x)|E(t,x)|\, dx, 
	\end{equation}
	and, by the sublinearity of $a$ in \eqref{eq:sublinearity}, $a\vphi$ is linearly bounded so that the integral
	$\int_0^\infty a(x) |\vphi(x)| f_2(s,x) \, dx$ can be bounded on $(0,T)$. Hence, using \eqref{eq:bound E+} together with \eqref{eq:pregronwall-Ep-regular}-\eqref{eq:inter-w-1} and \eqref{eq:gronwall condition} provides \textit{a priori} all the estimates needed to close the Gronwall loop. However, the function $\vphi$ we have constructed does not fulfill the requirements in Lemma \ref{eq:pregronwall-Ep-regular}; thus, a regularization argument is needed, for which  some care at $x$ close to the origin is required. The control of the nucleation rate, together with \eqref{eq:control_on_E+(0)}, will provide us with a suitable bound.
	
	\begin{lemma}\label{lem:pre-gronwall-1}
		Let assumption \eqref{H8a} or \eqref{H8b} hold true, and let $\vphi$ be defined in Lemma \ref{lem:vphi} above. Then, there exists $C>0$ such that
		\begin{multline}
		\nonumber
		\int_{0}^\infty \vphi(x)|E(t,x)|\, dx  \leq  \int_{0}^\infty \vphi(x)|E(0,x)|\, dx  + C\int_0^t \int_{0}^\infty \vphi(x)|E(s,x)| \, dx\, dt \\
		+ C\int_0^t |w(s)|   \, dt + C \int_0^{t} |E(s,0)| \, dt\,. 
		\end{multline}
	\end{lemma}
	
	\begin{proof}
		To substitute the function $\vphi$ defined in Lemma \ref{lem:vphi} into Eq.~\eqref{eq:pregronwall-Ep-regular} from Lemma \ref{lem:newlem_E}, we need to truncate its support around zero.
		For each $R>1$, let $\chi_R \in\Dc(\Rb)$ with $0\leq \chi_R\leq 1$, such that $\chi_R=1$ on $(1/R,\infty)$, with support in $[1/2R,\infty)$, $|\chi_R'|\leq 4R$ on $(1/2R,1/R)$. Define $\vphi_R=\vphi \chi_R$ on $(0,\infty)$. 
		Set $R>\bar{x}>1/R$. By Lemma \ref{lem:newlem_E}, using that $\vphi_R \leq \vphi$, that $f_2\in L^\infty((0,T);L^1((1+x)dx))$ and that $a\vphi$ is bounded on $(0,\bar{x})$, we get, 
		\begin{multline} 
		\label{eq:inter E+ vphiR}
		\int_0^\infty \!\!  \vphi_R(x)|E(t,x)|\, dx  \leq   \int_0^\infty \! \! \vphi(x) |E(0,x)|\, dx \\
		+ \int_0^t \! \int_0^\infty \! \! \partial_x[v_1(s,x))\vphi_R(x)]  |E(s,x)| \, dx \, dt 
		\\
		+ (\|a\vphi\|_{L^\infty(0,\bar{x})}+ K_r\vphi(\bar{x})) \|f_2\|_{L^\infty(L^1((1+x)dx))} \int_0^t |w(s)|\, dt .
		\end{multline}
		where $K_r$ follows from \eqref{eq:sublinearity}. Using Lemma \ref{lem:vphi} we deduce that there exists a constant $C>0$ such that
		\begin{multline}
		\label{eq:32}
		\partial_x[v_1(t,x))\vphi_R(x)] = \partial_x[v_1(t,x) \vphi(x)]\chi_R(x) +  v_1(t,x) \vphi(x) \chi_R'(x) \\
		\leq C \vphi(x)\chi_R(x) +4R {|} v_1(t,x){|} \vphi(x)  \mathbf 1_{(0,1/R)}(x).	
		\end{multline}
		Thus, from \eqref{eq:32},
		\begin{multline*}
		\int_0^{\infty} \partial_x[v_1(s,x))\vphi_R(x)]  |E(s,x)| \, dx \leq C\int_0^{\infty} \vphi(x)|E(s,x)|\, dx\\+ 4 \|a\vphi\|_{L^\infty(0,\bar{x})}\|u_1(t)-\Phi\|_{L^\infty((0,T)\times(0,\bar{x}))} R \int_0^{1/R}|E(s,x)|\, dx
		\end{multline*}
		Introducing the equation above into Eq.~\eqref{eq:inter E+ vphiR} and letting $R\to \infty$ we obtain the desired estimate.		
	\end{proof}
	
	\begin{proof}[Proof of Theorem \ref{thm:uniqueness}]
		We may now finish the proof of uniqueness. By Lemma \ref{lem:pre-gronwall-1}, equations \eqref{eq:inter-w-1}, \eqref{eq:control_on_E+(0)} and \eqref{eq:bound E+} combined with Gronwall's lemma we have
		\begin{equation*}
		|w(t)|+|E(t,0)| + \int_{0}^{\infty} \vphi(x) |E(t,x)|\, dx
		\leq C\left( |E(0,0)| +   \int_0^{\infty} \vphi(x)|E(0,x)|\, dx   \right) e^{CT}
		\end{equation*}
		and we conclude the proof of Theorem \ref{thm:uniqueness} by taking $f^{\rm in}_1=f^{\rm in}_2$, so that $E(0,x)=0$ for all $x\ge 0$  and then $u_1(t)=u_2(t)$, and $f_1(t,\cdot)=f_2(t,\cdot)$. 
	\end{proof}

\subsection{Criteria for global and local solutions}\label{sec:global}

In this section we prove Theorem \ref{thm:global local}, stating criteria both for existence of global solutions and for existence of local solutions for which $u$ reaches the value  $\Phi_0$ in finite time. Recall that for a solution on $[0,T)$, by Lemma \ref{lem:moments equations} we have that
\begin{equation}
\label{handy_form}
\frac{d u(t)}{dt} =  \int_0^\infty a(x) (\Phi(x)-u(t))f(t,x)\, \, dx
\end{equation}
is continuous on $[0,T)$. We exploit this formulation in the current  section. For that aim let us introduce
\[
\Phi^{sup}:=\sup_{x\ge 0}\Phi (x),\quad \Phi_{inf}:=\inf_{x\ge 0}\Phi (x).
\]
Note that $0\le \Phi_{inf}\le  \Phi_0 \le \Phi^{sup}$, where $\Phi^{sup}$ need not be finite.
\begin{lemma}
 Let the rates and initial datum satisfy the assumptions of Theorem \ref{thm:existence} and let $f\in \mathcal C([0,T);w-L^1((0,\infty),(1+x)\, dx))$ be a solution to the Lifshitz-Slyozov equation on $[0,T)$. Then the following assertions hold true: 
	\begin{enumerate} [align=left, leftmargin=2pt, labelindent=\parindent,label=\textup{(\arabic*)},itemsep=3pt]
		\item  For $t\in[0,T)$, $u(t)\ge \Phi^{sup}$ implies that $\dot u(t) \le 0$.
		\item Assume that there is some $\bar t \ge 0$ such that $u(\bar t) \in  [\Phi_{inf},\Phi^{sup}]$. We have $u(t)\in [\Phi_{inf},\Phi^{sup}]$ for every $t\in [\bar t,T)$.
		\item Assume that the solution is global ({\it i.e.} $T=\infty$). Then, provided that $\rho \ge \Phi_{inf}$, both $\liminf_{t \to \infty} u(t)$ and $\limsup_{t \to \infty} u(t)$ belong to  $[\Phi_{inf},\Phi^{sup}]$.
	\end{enumerate}
\end{lemma}

\begin{proof}
All statements follow easily from Eq.~\eqref{handy_form}.
\end{proof}
\begin{proof}[Proof of Theorem \ref{thm:global local}]
	To prove the first point we argue by contradiction. Let $f\in \mathcal C([0,T);w-L^1((0,\infty),(1+x)\, dx))$ be a solution with $T<\infty$. Since $\Phi(x) \geq \Phi_0$, 
	\begin{equation}
	\nonumber
	\frac{d  (u(t)-\Phi_0)}{dt} =   \frac{d u(t)}{dt} \geq -(u(t)-\Phi_0)\int_0^\infty a(x)f(t,x)\, dx.  
	\end{equation}
	This entails 
	\[u(t)-\Phi_0 \ge (u(0)-\Phi_0) \exp\left(- \int_0^t\int_0^\infty a(x)f(t,x)\, dx\right)\,.\]
	Note that the integral of $af$ is bounded on bounded time intervals. This is due to the bound \eqref{eq:sublinearity}: mass conservation controls the linear part, while the boundedness of $\mathfrak{n}$ and $u$ controls the constant part. Thus,
	\[u(t)-\Phi_0 \ge (u(0)-\Phi_0) \exp\left(- T\sup_{t\in(0,T)}\int_0^\infty a(x)f(t,x)\, dx\right)>0\,.\]
	This implies that $u(t) > \Phi_0$ {for every $t\in [0,T)$}. Then, if we assume that $T<\infty$, we deduce that $u(T^-)>\Phi_0$. This enables us to apply Theorem \ref{thm:existence} and extend this solution to a larger time interval, which contradicts our premise.

	The second point is also proved by contradiction. Assume that we have a global solution $f\in \mathcal C([0,\infty);w-L^1((0,\infty),(1+x)\, dx))$. We start by deriving an upper bound on $u(t)$. Thanks to the convexity of $\Phi$ we have 
	\[ \Phi(x)  \leq  \Phi_0 + \frac{\Phi(z) - \Phi_0}{z}x\]
	for all $0<x<z$. 
	Now say that the support of $f^{\rm in}$ is contained in $[0,x_0]$. Note that 
	\[X(t;0,x_0) \leq x_0 + \overline a \rho t \]
	and hence the support of $f(t,\cdot)$ is contained in $[0,z(t)]$ for every $t\ge 0$,
	where we have denoted $z(t) :=  x_0 + \overline a t$. Hence we have
	\begin{multline}
	\nonumber
	\frac{d u(t)}{dt} = \int_0^{z(t)}\!  a(x)(\Phi(x)-u(t)) f(t,x)\, dx \\
	\leq \int_0^{z(t)} \! \left[ \Phi_0-u(t) + \frac{\Phi(z(t))-\Phi_0}{z(t)} x\right]\! a(x) f(t,x)\, dx\,.
	\end{multline}
	 We remark that $\Phi(z(t)) \leq \Phi(x_0) <\Phi_0$, therefore we have 
	\begin{equation}
	\nonumber
	\frac{d u(t)}{dt} 
	\leq - \frac{\Phi_0-\Phi(x_0)}{z(t)} \underline a \int_0^{z(t)} x f(t,x)\, dx\,.
	\end{equation}
	Using mass conservation,
	\begin{equation}
	\nonumber
	\frac{d u(t)}{dt}
	\leq - \frac{\Phi_0-\Phi(x_0)}{z(t)} \underline a (\rho-u(t)) \leq 0\,.
	\end{equation}
	Hence $u$ decreases and $\rho-u(t) \ge \rho - u(0)$. Then we conclude that
	\begin{equation}
	\nonumber
	\frac{du(t)}{d t}
	\leq -K\frac{1}{z(t)} = - \frac{K}{x_0 + \overline a t}
	\end{equation}
	where $K= \underline a (\Phi_0-\Phi(x_0))(\rho-u(0))$. Integrating the differential inequality we obtain
	\begin{equation}
	    \label{eq:logcontrol}
	    u(t) \leq u(0) - \frac{K}{\overline a} \ln\left(1+\frac{\overline a}{x_0}t\right)\, 
	\end{equation}
	for every $t\ge 0$. There is a unique value $\tilde T<\infty$ such that the right-hand side of \eqref{eq:logcontrol} equals $\Phi_0$; then our premise is not compatible with Definition \ref{def:inflow}.
\end{proof}

\section{Annex: Linear transport equations with degenerate transport fields}

\label{sec:annex}
The purpose of this section is to study the linear continuity equation having the following form:
\begin{equation} 
\label{eq:linearproblem}
\left\{
 \begin{array}{l}
 \ds \frac{\partial f(t,x)}{\partial t} + \frac{\partial [v(t,x)f(t,x)]}{\partial x} = 0  \vphantom{\int} \, , \quad t\in (0,T)\,, \ x\in(0,\infty)\,,\\[0.8em]
\lim_{x\to 0^+} v(t,x)f(t,x) = G(t)
\, , \quad t\in (0,T), 
 \\[0.8em]
 f(0,x)=f^{in}(x)\,,\ \quad x\in (0,\infty)\:.
 \\[0.8em]
 \end{array}
 \right.
\end{equation}
For this problem we are given a bounded continuous function $G$ on $[0,T)$. We are interested in a class of transport fields $v(t,x)$ that yield inflow behavior and that may eventually be degenerate at the origin. More specifically, we assume that $v$ can be factorized in the following way
\begin{equation}\label{eq:v=aw}
v(t,x)=a(x)w(t,x) 
\end{equation}
for all $(t,x)\in \Omega_T$, with the following assumptions:
\begin{align}
& \text{There exists } K>0 \text{ such that } |v(t,x)| \leq K(1+x) \text{ for all } (t,x)\in \Omega_T\,. \tag{A1} \label{A1}\\
&\partial_x v \in L^\infty((0,T)\times(1/R,\infty)),\ \forall R>0. \tag{A2} \label{A1p} \\
& a\in \mathcal C^0([0,\infty))\cap \mathcal C^1(0,\infty)\, \text{verifies that}\, a(x)>0\, \forall x>0. \tag{A3}\label{A2}\\
& w \!\in\! \mathcal C^0([0,T)\!\times\![0,\infty)),\, \partial_x w \!\in\! \mathcal C^0((0,T)\!\!\times\!\!(0,\infty))\! \cap \! L^\infty((0,T)\!\times\! (\frac{1}{R},R)) \, \forall R>0 \tag{A4} \label{A3} \\
&  \tfrac 1 a \in L^1(0,1) \text{ and } \lim_{x\to+\infty} \int_0^x \tfrac 1 {a(y)} dy = +\infty \,. \tag{A5} \label{A4} \\
&  \partial_x w \in L^1((0,T)\times(0,1)) \,.\tag{A6} \label{A5}
\end{align}
We do not strive for optimality in our assumptions; rather, we present an assumption set that is compatible with that of the main text. In that regard, here we assume inward flow by the condition: There exists $\delta>0$ and $x_0>0$ such that
\begin{equation}\tag{A7}\label{A6}
\forall(t,x)\in[0,T)\times(0,x_0),\  w(t,x) \geq \delta\,. 
\end{equation}
This condition is readily entailed by a property like $w(t,0)>0$ for all $t\in[0,T)$ thanks to the continuity; thus, it is not restrictive at all.
\begin{remark}[Notations for partial derivatives] For functions of two variables like $v(t,x)$, 
during this section we always refers to the function $\partial_x v$, as the partial derivative in the second variable. We shall use unambiguous expression like $\partial_y v(t,y) = (\partial_x v)(t,y)$ at some places. When dealing with characteristics, we take the usual convention that the variable that appears in the denominator of the partial derivative operator indicates in which variable the derivative is to be taken, consistently with the independent variable being used.
\end{remark}

The idea is that we expect to have $a(0)=0$ and thus we factor out the degeneracy of the transport field at the origin. It is easy to see that under our running assumptions \eqref{H1}--\eqref{H4} in the main text, the transport field $v$ with $w(t,x):=u(t)-\Phi(x)$ fulfills the former set of conditions and therefore the theory applies to the linear problem that is considered in Sections \ref{sec:inflow} and \ref{sec:nonlinear}, with $G(t):=\mathfrak n(u(t))$ and $u(t)>\Phi_0$. 
\begin{theorem}
\label{th:compilation}
Let $f^{in}\ge 0$ satisfies \eqref{H6} and let all the assumptions \eqref{A1}-\eqref{A6} of this section be satisfied. Then, there exists a unique solution $f$ to \eqref{eq:linearproblem}, that is: \begin{enumerate}
\item
\label{itemreg}
$f\in L^\infty\! \left((0,T);L^1((0,\infty),(1+x)dx)\right)\cap \mathcal{C}\left([0,T);w\!-\!L^1((0,\infty),(1+x)dx)\right)$.
\item
\label{itemmoment}
For all $\vphi\in \mathcal{C}^1_c([0,T)\times[0,\infty))$, there holds that
    \begin{multline} 
    \int_{0}^{T} \int_0^\infty \left(\partial_t \vphi(t,x) + v(t,x)\partial_x\vphi(t,x)\right) f(t,x)\, dx \, dt\\
    + \int_{0}^{T} \vphi(t,0)G(t) \, dt  + \int_0^\infty \vphi(0,x)f^{\rm in}(x) \, dx = 0\,.
    \end{multline}  
\item
\label{eq:trace}
There holds that $\lim_{x\to 0} v(t,x)f(t,x) = G(t).$ 
\end{enumerate}
Moreover, for any real function $h$ locally bounded on $(0,\infty)$ such that $h'\in L^\infty(0,\infty)$, the solution $f$ satisfies
\begin{multline}\label{eq:moment_formulation}
\int_0^\infty h(x) f(t,x)\,dx = \int_0^\infty h(x)f^{\rm in}(x)\,dx +\int_0^t\int_0^\infty v(s,x) h'(x) f(s,x)\,dx\, ds \\
+h(0)\int_0^t G(s)\, ds \,.
\end{multline}
\end{theorem}

\subsection{Characteristic curves and the reparametrization strategy}
\label{Ann:1}
Due to the lack of Lipschitz regularity, the analysis of the characteristic curves will be tackled thanks to a reparametrization of the flow through a diffeomorphism, leading to a positive lower bound of the time derivative of the reparametrized characteristic curves at the boundary $x=0$. Let us start by introducing the characteristic curves $s\mapsto X(s)$ associated with the generic field $v$.

\begin{lemma} \label{lem:ode-X}
    For any $(t,x)\in\Omega_T$, there exists a unique maximal solution to 
    \begin{equation} \label{eq:characteristics curves}
    \left\{
        \begin{array}{l}
        \displaystyle \dfrac{\partial X(s;t,x) }{\partial {s}} = v(s,X(s;t,x)) \, , \\[0.8em]
        \displaystyle X(t;t,x)=x 
    \end{array}
    \right.
    \end{equation}
    with maximal interval $\Sigma_{t,x}$. Moreover, the following properties hold true:
    \begin{enumerate} [align=left, leftmargin=2pt, labelindent=\parindent,label=\textup{(\arabic*)},itemsep=3pt]
	\item For any $(t_0,x_0)\in \Omega_T$ and $s_0 \in \Sigma_{t_0,x_0}$, there exists a neighborhood  of $(s_0,t_0,x_0)$ in $ \Sigma_{t_0,x_0} \times \Omega_T$ such that  $(s,t,x)\mapsto X(s;t,x)$ is well defined and continuously differentiable;
	\item The semigroup property $X(t;s,X(s;t,x))=x$ is satisfied for every $s\in \Sigma_{t,x}$;
    \item For every $s\in \Sigma_{t,x}$ we have  
    \begin{equation} \label{eq:characteristics curves - derivatives in x and t}
	\begin{array}{l}
	    \ds \dfrac{\partial  X(s;t,x)}{\partial x}  \coloneqq J(s;t,x) = \exp \left( - \int_s^t  (\partial_x v)(\tau,X(\tau;t,x))\, d\tau \right)\, , \\[0.8em]
	    \ds \dfrac{\partial X(s;t,x)}{\partial t}  = -v(t,x)J(s;t,x)\, ;
    \end{array}
    \end{equation}
    \item There exists a positive constant $C(K,T)$, which only depends on $T$ and $K$ from \eqref{A1}, such that  
    \begin{equation} \label{eq:characteristics curves - uniform bounds}
	    X(s;t,x) + \left| \dfrac{\partial X(s;t,x)}{\partial {s}} \right|  \leq C(K,T)(1+x) 
    \end{equation}
    for all $(t,x)$ in $\Omega_T$ and $s$ in $\Sigma_{t,x}$. As a consequence, each characteristic curve has a finite limit in $[0,\infty)$ at the end points of $\Sigma_{t,x}$.
    \end{enumerate}
\end{lemma}

\begin{proof}
Existence, uniqueness and maximality readily follow from the Cauchy--Lips\-chitz theory for ordinary differential equations, since both $v$ and $\partial_x v$ are continuous. Point 1 is a classical regularity result, see \cite[Chap. V Cor. 3.3]{Hartman82}. Point 2 follows from uniqueness. The derivatives in Point 3 are computed in a standard fashion (see the textbook above). Finally, point 4 is a consequence of \eqref{A1} and Gronwall's lemma; this prevents the blow-up of the characteristics at the end points of $\Sigma_{t,x}$.
\end{proof}  

Similarly to the main text, we define for all $(t,x)\in \Omega_T$ the time
\[\sigma_t(x):=\inf \Sigma_{t,x}\:.\]
This represents the backward lifetime of the characteristic, also introduced in \cite{Boyer05}.

\begin{lemma} \label{lem:either-or property for sigma}
    Let $(t,x)\in \Omega_T^*$. If $\sigma_t(x)=0$ then  $X(s;t,x)>0$ for all $s$ in $(0,t)$. Otherwise, if $\sigma_t(x) > 0$ then $\lim_{s \to \sigma_t(x)^+} X(s;t,x) = 0$. 
\end{lemma}

\begin{proof}
    Since the characteristics take values in $(0,\infty)$, the first statement follows from the definitions of $\Sigma_{t,x}$ and $\sigma_t(x)$. In the case $\sigma_t(x)>0$, since a characteristic curve has a finite limit at the lower end of $\Sigma_{t,x}$, this limit is either positive or zero. But, if the limit is positive (say $\bar x$), thanks to the Cauchy-Lipschitz theory we can construct a prolongation of the characteristic curve in a neighborhood of $(\sigma_t(x),\bar x)$, which contradicts the definition of $\sigma_t(x)$. Therefore the limit at $\sigma_t(x)$ vanishes. 
 \end{proof}

 Now we introduce the reparametrization. Thanks to \eqref{A2} and \eqref{A4} we define   
\[A(x) := \int_0^x \frac{1}{a(y)}\, dy \quad \mbox{for all}\, x>0 .\]
Clearly, $A$ is an increasing $\mathcal C^1$-diffeomorphism from $(0,\infty)$ into itself. Note that both $A$ and $A^{-1}$ might be extended continuously at $0$ by $A(0)=A^{-1}(0)=0$ and we have $(A^{-1})'(x)=a(A^{-1}(x))$ for all $x>0$ and we can set $(A^{-1})'(0)=a(0)$. Then, we define the reparametrized transport field by
\[V(t,x) = w(t,A^{-1}(x))\]
for each $(t,x)$ in $\Omega_T$. Note that this reads $V(t,x) = u(t) - \Phi \circ A^{-1}(x)$ for the linear problem in the main text. The associated trajectories are given by the following result:

\begin{lemma}
	For any $(t,y)\in\Omega_T$, there exists a unique maximal solution to 
    \begin{equation} \label{eq:edo-B}
	\left\{
	\begin{array}{l}
	    \displaystyle \frac{\partial  B(s;t,y) }{\partial {s}} = V(s,B(s;t,y))\,,\\[0.8em]
	    B(t;t,y) = y
	\end{array}
	\right.
    \end{equation}    
	with maximal interval $\tilde \Sigma_{t,y}$.  Moreover, the following properties hold true:
	\begin{enumerate} [align=left, leftmargin=2pt, labelindent=\parindent,label=\textup{(\arabic*)},itemsep=3pt]
	\item For any $(t_0,y_0)\!\in \Omega_T$ and $s_0 \!\in \tilde \Sigma_{t_0,y_0}$, there exists a neighborhood  of $(s_0,t_0,y_0)$ in $\tilde \Sigma_{t_0,y_0} \times \Omega_T$ such that  $(s,t,y)\mapsto B(s;t,y)$ is well defined and continuously differentiable;
	\item The semigroup property $B(t;s,B(s;t,y))=y$ is satisfied for every $s\in \tilde \Sigma_{t,y}$.
	\item For any $(t,x)\in\Omega_T$, we have 
    \begin{equation} \label{eq:compositio}
	    \tilde \Sigma_{t,A(x)}=\Sigma_{t,x}\ \text{and}\ B(s;t,A(x))=A(X(s;t,x))\,,\ \text{for any}\ s\in \Sigma_{t,x}\,.
    \end{equation}
     \item For every $s\in \tilde \Sigma_{t,y}$ we have
    \begin{equation} \label{eq:derivatives of B}
    \begin{array}{rcl}
	    \ds \dfrac{\partial B(s;t,y) }{\partial y}  & \coloneqq &   \ds I(s;t,y) = \exp\left(\int_s^t \left(a \, \partial_x w\right)(\tau,A^{-1}(B(\tau;t,y))) \, d \tau \right)\, ,\\[0.8em] 
		\ds \dfrac{\partial  B(s;t,y) }{\partial t} & = &  -V(t,y) I(s;t,y)\, ; 
	\end{array}
	\end{equation} 
	\end{enumerate}
    Here and in what follows we understand that $(a\, \partial_x w)(t,x)=a(x) \, (\partial_x w)(t,x)$ in order to ease some formulas. 
\end{lemma}

The advantage with respect to the characteristics given by the original transport field is that  we have factored out the degeneracy of the transport field at the origin; in such a way we avoid Peano-like phenomena \cite{Crippa}. Take for instance \eqref{eq:v=aw}, given the rates $a(x)=x^{1/3},\, b(x)=x^{1/2}$ the transport field vanishes at the origin like $x^{1/3}$ -given that the behavior about the origin is not Lipschitz we do not expect uniqueness for the forward characteristics; however, we find that $V(t,x)=u(t)-(\frac{2}{3}x)^{3/4}$, which does not have such pathological behavior and therefore makes the associated integral curves easier to work with.

\begin{proof}
Fix $(t,y)\in\Omega_T$. Since both $V$ and $\partial_x V$ are continuous, there exists a unique solution $s\mapsto B(s;t,y)$ to \eqref{eq:edo-B},  defined on a maximal interval $\tilde \Sigma_{t,y}\subset[0,T)$ containing $t$ and with range in $(0,\infty)$. Taking into account that 
\[\frac{\partial V(t,y)}{\partial y}= \left(a\, \partial_x w\right)(t,A^{-1}(y))
\,,\ 
\text{for all}\ y>0\,,\]
all the stated properties follow easily as in Lemma \ref{lem:ode-X}, except maybe \eqref{eq:compositio}. We now prove \eqref{eq:compositio}. Let $(t,x)\in\Omega_T$. First, $s\mapsto A(X(s;t,x))$ is a solution to \eqref{eq:edo-B} with $A(X(t;t,x))=A(x)$. Thus $\Sigma_{t,x} \subseteq \tilde \Sigma_{t,A(x)}$ and $B(s;t,A(x))=A(X(s;t,x))$ for all $s\in \Sigma_{t,x}$. Define  $Y(s;t,x)=A^{-1}(B(s;t,A(x)))$ for all $s\in \tilde \Sigma_{t,A(x)}$. Then $Y$ is a solution to the original equation \eqref{eq:characteristics curves} with $Y(t;t,x)=x$, thus $\tilde \Sigma_{t,A(x)}\subseteq \Sigma_{t,x}$. Therefore, $\tilde \Sigma_{t,A(x)}=\Sigma_{t,x}$ and \eqref{eq:compositio} holds.
\end{proof}

\begin{remark}
\label{rem:no cross}
We will frequently use in the proofs below that the derivatives of $X$ and $B$ with respect to their third argument are positive. In other words, uniqueness ensures that characteristics cannot cross and hence we have the following monotonicity property: given $x<y$, then for all $s\in \Sigma_{t,x}\cap \Sigma_{t,y}$ we have $X(s;t,x)<X(s;t,y)$ and $B(s;t,A(x))<B(s;t,A(y))$.
\end{remark}

The control of the time derivative of $X$ at the boundary $x=0$ is stated next, thanks to the characteristics $B$.

\begin{lemma} \label{lem:invariance}
	For every $(t,x)\in\Omega_T$ and $\tau \in \Sigma_{t,x}$ such that $X(\tau;t,x)<x_0$, with $x_0$ defined in \eqref{A6}, the following holds:
	\begin{enumerate} [align=left, leftmargin=2pt, labelindent=\parindent,label=\textup{(\arabic*)},itemsep=3pt]
	\item The map $s\mapsto X(s;t,x)$ is an increasing $\mathcal C^1$-diffeomorphism from $(\sigma_t(x),\tau)$ to $(0,X(\tau;t,x)),$
	\item for every $s\in(\sigma_t(x),\tau)$ we have the lower bound
	\[\frac{\partial A(X(s;t,x))}{\partial {s}} = \frac{\partial B(s;t,A(x)) }{\partial {s}} \geq  \delta\,.\]
	\end{enumerate}
	Moreover, for every $(t,x)\in\Omega_T$, there holds that 
	\begin{enumerate} [align=left, leftmargin=2pt, labelindent=\parindent,label=\textup{(\arabic*)},itemsep=3pt]  \setcounter{enumi}{2}
	\item $\Sigma_{t,x}=(\sigma_t(x),T)$ if $t\in(0,T)$,  while $\Sigma_{0,x}=[0,T)$,
	\item for every $s\in [t,T)$ we have $X(s,t,x)\geq \min(x,x_0)$.
	\end{enumerate}
\end{lemma}

\begin{proof} 
Let $(t,x)\in \Omega_T$, we may rewrite \eqref{eq:characteristics curves} as
\begin{equation*}
\frac{\partial X(s;t,x)}{\partial s} = a(X(s;t,x))\, w\left(t,(X(s;t,x))\right)\,.
\end{equation*}
Since $a$ is positive, the flow verifies $a(z)w(t,z)>a(z)\delta>0$ for all $(t,z)\in[0,T)\times(0,x_0)$, which shows that the interval $(0,x_0)$ is negatively invariant. In other words, if there exists $\tau \in \Sigma_{t,x}$ such that $X(\tau;t,x)<x_0$ then, $X(s;t,x)<x_0$  and
\[\frac{\partial B(s;t,A(x)) }{\partial s} = V(s,B(s;t,A(x))) = w(s,X(s;t,x)) \ge \delta\,,\]
for all $s\in(\sigma_t(x),\tau)$, which proves the second point. Thus, the first point follows directly from this fact and using that $A$ is increasing with $B(s;t,A(x))=A(X(s;t,x))$. We then prove the last two points. As  $(0,x_0)$ is negatively invariant, and since the flow is positive on $(0,x_0)$, this also proves that $(0,x)$ is negatively invariant for each $x\in(0,x_0)$. We claim that $(x,\infty)$ is positively invariant for all $x\in(0,x_0]$. This can be proved arguing by contradiction: Let $y\in(x,\infty)$ and $t\in[0,T)$; if there exists $s>t$ such that $X(s;t,y)\le x$, then for all times $\tau\leq s$, we have that $X(\tau,t,y)\le x$ because $(0,X(s;t,y))$ is negatively invariant. We deduce that $X(t;t,y)=y \le x$, which contradicts the premise and yields our claim. In fact, this argument readily entails $X(s;t,x)\ge \min(x_0,x)$ for all $(t,x)\in\Omega_T$ and $s\in \Sigma_{t,x}\cap(t,T)$. We conclude thanks to the lower bound and remarking that the Cauchy-Lipschitz theory allows to prolongate solutions up to time $T$ by the regularity of $v$.
\end{proof}

The following technical lemma turns out to be crucial to bound the derivatives of $B$, see equations in \eqref{eq:derivatives of B}. This result also shows that assumption \eqref{H4} in the main text, $\Phi' \in L^1(0,1)$, is close to be optimal to prevent concentration in finite time. This is mirrored by Assumption \eqref{A5} in this general framework.
\begin{lemma} \label{lem:decomposition a phi}
	 Let $\delta>0$ and $x_0$ be given by Assumption \eqref{A6}.
	 For all $t\in(0,T)$, $(s,\tau,s_0)\in(0,t]^3$, $x_1\in(0,x_0]$ and  $x>0$, if $\sigma_\tau(x) \leq s \leq s_0$ and $X(s_0;\tau,x)<x_1$, then there holds that 
    \begin{equation*}
	\int_s^{s_0} \left|  \left(a\, \partial_x w\right)(r,X(r;\tau,x))\right| \, dr  \leq  \frac 1 \delta \int_0^t \int_0^{x_1} \left|\partial_y w(r,y)\right| \, dy dr\:.
	\end{equation*}
\end{lemma}
\begin{proof}
We notice that thanks to Lemma \ref{lem:invariance} we have  \[X(r;\tau,x)<X(s_0;\tau,x)<x_1<x_0 \quad \mbox{for all}\quad r\in(\sigma_\tau(x),s_0)\,.\]
Hence, using Eq.~\eqref{eq:characteristics curves} and \eqref{A6},
\begin{multline}
\label{eq:210}
\int_s^{s_0}  \left|  \left(a\, \partial_x w\right)(r,X(r;\tau,x))\right| \, dr  =  \int_s^{s_0} \left| \dfrac{(\partial_x w)(r,X(r;\tau,x))}{w(r,X(r;\tau,x))} \dfrac{\partial X(r;\tau,x)}{\partial r}  \right| \, dr\\
\le \frac 1 \delta \int_s^{s_0} \left|(\partial_x w)(r,X(r;\tau,x)) \dfrac{\partial X(r;\tau,x) }{\partial r} \right| \, dr \\
= \frac 1 \delta \int_{X(s;\tau,x)}^{X(s_0;\tau,x)} \left|(\partial_x w)(\theta(y),y)\right| \, dy 
\end{multline} 
where $\theta \colon (0,X(s_0;\tau,x)) \to (\sigma_\tau(x),s_0)$ is the inverse map of the $\mathcal C^1$-diffeomorphism $s\mapsto X(s;\tau,x)$ on $(\sigma_t(x),s_0)$ -here we use the first point of Lemma \ref{lem:invariance}.
Then we notice that $0<X(s;\tau,x)<X(s_0;\tau,x)<x_1$ and $s_0<t$, so that we can write 
\begin{multline*}
\frac 1 \delta \int_{X(s;\tau,x)}^{X(s_0;\tau,x)} \left|(\partial_x w)(\theta(y),y)\right| \, dy = \frac{1}{\delta}  \int_{X(s;\tau,x)}^{X(s_0;\tau,x)} \int_s^{s_0} \left|\partial_y w(r,y)\right| \mathbf 1_{\theta(y)=r}\, dr dy \\
\le \frac 1 \delta \int_{0}^{x_1} \int_0^t \left|\partial_y w(r,y)\right| \, dr dy\,.
\end{multline*}
\end{proof}

%\medskip

\subsection{Diffeomorphism through the characteristic curves}
\label{Ann:2}

Let us define, similarly to the main text, for each $t\in[0,T)$:
\[x_c(t)=\inf \{x>0\,|\,\sigma_t(x)=0\}\:.\]

\begin{lemma}  \label{lem:separation by x_c(t)}
    For each $t\in(0,T)$ we have the following properties:
	\begin{enumerate} [align=left, leftmargin=2pt, labelindent=\parindent,label=\textup{(\arabic*)},itemsep=3pt]
    \item The value $x_c(t)$ is finite and positive,
    \item $\sigma_t$ is a nonincreasing map which is positive on $(0,x_c(t))$,
    \item $\sigma_t$ vanishes on $(x_c(t),\infty)$.
    \end{enumerate}   
    Moreover, for $t=0$ we have: $x_c(0)=0$ and $\sigma_0$ is constantly equal to zero. 
\end{lemma}

\begin{proof}

\textit{Step 1. Proof that $\sigma_t$ is nonincreasing.} Let $t\in[0,T)$ and $0<x<y$. If $\sigma_t(y)=0$ then $\sigma_t(x)\geq \sigma_t(y)=0$ by definition.  Assume now that $\sigma_t(y)>0$; we prove the monotonicity in this case by a contradiction argument. Therefore, let us assume that $\sigma_t(y)>\sigma_t(x)$. By Remark \ref{rem:no cross}, $X(s;t,x) < X(s;t,y)$ for all $s \in (\sigma_t(x)\vee\sigma_t(y),T)$ and we have 
\[0\leq \lim_{s \to \sigma_t(x)\vee\sigma_t(y)} X(s;t,x) \leq \lim_{s \to \sigma_t(x)\vee\sigma_t(y)} X(s;t,y)\,.\] 
Now thanks to our assumption $\sigma_t(y)>\sigma_t(x)$ and Lemma \ref{lem:either-or property for sigma}, we obtain
\[
0\le X(\sigma_t(y);t,x)\leq \lim_{s\to \sigma_t(y)}X(s;t,y)=0\,,
\] 
which entails $X(\sigma_t(y);t,x)=0$. But this contradicts the definition of the maximal interval $\Sigma_{t,x}$. Thus, $\sigma_t(x)\geq \sigma_t(y)$ as desired.

\smallskip

\textit{Step 2. Proof that $x_c(t)$ is finite, \textit{i.e.} that $\set{x>0}{\sigma_t(x)=0}$ is not empty.} Let $t\in[0,T)$, $y>0$ and $x=X(t;0,y)>0$. By the semigroup property, we have that $X(s;t;x)=X(s;0,y)$ for all $s\in(\sigma_t(x),T)$. Since {$\Sigma_{0,y}=[0,T)$}, the trajectory $s\mapsto X(s;t,x)$ is defined
on $[0,T)$ and hence $\sigma_t(x)=0$. This proves that $\set{x>0}{\sigma_t(x)=0}$ is not empty, thus $x_c(t)$ is finite and nonnegative.
 
\smallskip

\textit{Step 3. Separation by $x_c(t)$.} Let $\{x^n\}\subset \set{x > 0}{\sigma_t(x)=0}$ be a nonincreasing sequence converging to $x_c(t)$. Since $\sigma_t$ is nonincreasing then $0\leq \sigma_t(x) \leq \sigma_t(x^n)=0$ for any $x>x^n$. Thus $(x_c(t),\infty) = \cup_{n\geq 1} (x^n,\infty) \subset \set{x > 0}{\sigma_t(x)=0}$. However, if $x$  is such that $\sigma_t(x)=0$ then $x\ge x_c(t)$ by definition, which proves that the former inclusion is an equality. Moreover, if $x_c(t)>0$, for any $x \in (0,x_c(t))$ we have $\sigma_t(x)>0$ by the construction of $x_c(t)$.

\smallskip

\textit{Step 4. $x_c(t)$ is positive.} Let $t\in(0,T)$. It is sufficient to construct some $x>0$ satisfying $\sigma_t(x)>0$, as then $x_c(t)\geq x>0$. Let $\delta >0$ and $x_0>0$ given by \eqref{A6}. Thus, for all $x\in(0,x_0)$ and $s\in (\sigma_t(x),t)$, since $X(t;t,x)=x<x_0$ we have, by Lemma \ref{lem:invariance}, that
\[
B(t;t,A(x)) - B(s;t,A(x)) \geq \delta (t-s),\quad \mbox{that is,}\quad B(s;t,A(x)) \leq A(x) - \delta(t-s)\,.
\]
Now we are ready to conclude by a contradiction argument. Contrary to what we want, suppose that $\sigma_t(x)=0$ for all $x\in(0,x_0)$. This enables us to compute the limit $\lim_{s\to 0} B(s;t,A(x)) \leq A(x) -\delta t$ for all $x\in(0,x_0)$. This entails that there exists $x_1\in(0,x_0)$ such that 
\[\lim_{s\to 0} B(s;t,A(x_1)) \leq A(x_1) -\delta t <0\,,\]
since $A$ is continuous and $A(0)=0$. This contradicts that, for all $s\in(0,t)$, $B(s;t,A(x_1))=A(X(s;t,x_1))>0$, after Lemma \ref{lem:either-or property for sigma}. Thus, there is some $x>0$ such that $\sigma_t(x)>0$. This concludes the proof of Lemma \ref{lem:separation by x_c(t)}.
\end{proof}

We address now the diffeomorphism given by $x\mapsto X(t;0,x)$. The first part of Prop. \ref{prop:diffeomorphism_merge} in the main text is a particular consequence of the following result (with $s=0$):

\begin{lemma} \label{prop:diffeo}
    For each $t\in(0,T)$ and $s\in[0,t)$ the following statements hold true: 
    \begin{enumerate} [align=left, leftmargin=2pt, labelindent=\parindent,label=\textup{(\arabic*)},itemsep=3pt]
    \item We have $X(t;s,0^+):=\lim_{x\to 0^+} X(t;s,x) \in(0,x_c(t))$.
    \item The map $x\mapsto X(t;s,x)$ is an increasing  $\mathcal{C}^1$-diffeomorphism from $(0,\infty)$ to $(X(t;s,0^+),\infty)$.
    \item 
    \label{it:s1}
    The semigroup property $X(t;\tau,X(\tau;s,0^+)) = X(t;s,0^+)$ holds $\forall \tau\in[s,T)$.
    \item \label{it:s2}
    We have $X(t;0,0^+)=x_c(t)$ and  $X(t;s,x_c(s))=x_c(t)$.
    \item 
    \label{it:bd} The bound $x_c(t)\le C(K,T)$ holds with the constant in \eqref{eq:characteristics curves - uniform bounds}.
    \end{enumerate}
\end{lemma}

\begin{proof}
In the sequel we use $\delta>0$ and $x_0>0$ given by \eqref{A6}. Let $t\in(0,T)$ and $s\in[0,t)$. Lemma \ref{lem:invariance} ensures that $\Sigma_{0,x}=[0,T)$ and $\Sigma_{s,x}=(\sigma_s(x),T)$ for all $x>0$. Then $x\mapsto X(t;s,x)$ is continuously differentiable on $(0,\infty)$ by the Cauchy--Lipschitz theory. Its derivative, given by~\eqref{eq:characteristics curves - derivatives in x and t}, is strictly positive. Thus, the map $x\mapsto X(t;s,x)$ is strictly increasing and then a diffeomorphism onto its image. We prove below that $\lim_{x\to \infty} X(t;s,x)=+\infty$ for any $s\in[0,t)$, that $X(t;0,0^+)=x_c(t)$, that $0<X(t,s,0^+)<x_c(t)$ for $s\in(0,t)$ and the semigroup properties. The bound on $x_c(t)$ in item {\it \ref{it:bd}} is a direct consequence of the bound \eqref{eq:characteristics curves - uniform bounds} at the limit $x\to 0$.
  
\smallskip
 
\textit{Step 1. Proof of $\lim_{x\to \infty }X(t;s,x)=\infty$.}
We fix $y>0$. Using the bound \eqref{eq:characteristics curves - uniform bounds} with $x=X(t;s,y)$ we deduce that
\[ 
X(\tau;t,X(t;s,y) )\leq C(K,T)(1+X(t;s,y))\,, \ \forall\, \tau \in \Sigma_{s,y}=(\sigma_s(y),T)=(\sigma_t(x),T).
\]
Setting $\tau=s$ we obtain
\[ y \leq C(K,T)(1+X(t;s,y))\,. \]
This concludes the proof by taking the limit $y \to \infty$.

\smallskip
 
 \textit{Step 2. Proof that $X(t;s,0^+)>0$}.
  Let $\{x_n\}$ a positive, decreasing sequence converging to zero.  By Lemma \ref{lem:invariance}, either $X(t;s,x_n)>x_0$ or $B(t;s,A(x_n))-A(x_n)\geq \delta (t-s)$. Thus, 
\[\lim_{n\to\infty}X(t;s,x_n)\geq \min(x_0,A^{-1}(\delta (t-s)))>0\,.\]
 
\textit{Step 3. Proof of $X(t;0,0^+)=x_c(t)$}. Let $t\in(0,T)$. We take a positive, nonincreasing sequence $\{x^n\}$ converging to zero. As $x\mapsto X(t;0,x)$ is monotonically  increasing we can define
\[\bar x := \lim_{x\to 0} X(t;0,x) =\lim_{x^n\searrow 0} X(t;0,x^n)\,.\]
 Note that $\sigma_t(X(t;0,x^n))=0$ and then $X(t;0,x^n)\ge \bar x \ge x_c(t)$, as $x_c(t)=\inf\set{x>0}{\sigma_t(x)=0}$. 

We prove that $\bar x = x_c(t)$ by a contradiction argument. Assume that 
\begin{equation}\label{eq:HA_x_xc}
\bar x>x_c(t). 
\end{equation} 
Let $y\in(x_c(t),\bar x)$; we have $\sigma_t(y)=0$ as $y>x_c(t)$. Since $y<\bar x\leq X(t;0,x^n)$ for any $n$, we obtain  
\[\lim_{s\to0} X(s;t,y) \leq \lim_{s\to0} X(s;t,X(t;0,x^n))=\lim_{s\to0} X(s;0,x^n)=x^n\,. \]
Passing to the limit $n\to \infty$ 
 we deduce that 
\[\lim_{s\to0} X(s;t,y)=0\,, \ \text{and hence}\ \lim_{s\to0} B(s;t,A(y))=0\]
by the continuity of $A$ at zero.

Consider now $y_1,y_2$ such that $x_c(t)<y_1<y_2<\bar x$. There holds that 
\begin{multline}
\label{eq:entre2B}
B(s;t,A(y_2))-B(s;t,A(y_1))=\int_{y_1}^{y_2} \frac{\partial B(s;t,A(z))}{\partial z}\, dz\\
=\int_{y_1}^{y_2} \left.\frac{\partial B(s;t,y)}{\partial y}\right|_{y=A(z)} \, \frac{dz}{a(z)}\,.
\end{multline}
We look for a lower bound on this quantity. Note that $\sigma_t(y_2)=0$ and that  $\lim_{s\to0} X(s;t,y_2)=0$. Thus, there exists $s_0$ such that $X(s_0;t,z)<X(s_0;t,y_2)<x_0$ for any $z\in(y_1,y_2)$. Now we use Lemma \ref{lem:decomposition a phi}: for any $s\in(0,s_0)$ we have that
\begin{equation}
\nonumber
\int_s^{s_0} \left|  \left(a\, \partial_x w\right)(r,X(r;t,z))\right| \, dr  \leq  \frac 1 \delta \int_0^t\int_0^{x_0} \left|\partial_\xi w (r,\xi)\right| \, d\xi dr
\end{equation}
 and from Eq.~\eqref{eq:derivatives of B} we deduce that for any $z\in(y_1,y_2)$,
\begin{multline}\label{eq:minor_bound_derivativeB}
\left.\frac{\partial B(s;t,y)}{\partial y}\right|_{y=A(z)}
\geq \exp\left(- \frac 1 \delta \int_0^t\int_0^{x_0} \left|\partial_\xi w(r,\xi)\right| \, d\xi dr \right.
\\
\left.
-\int_{s_0}^t\left|  \left(a\, \partial_x w\right)(r,X(r;t,z))\right| \, dr \right)
\coloneqq c_1(z)>0\,.
\end{multline}
Hence, as the lower bound on Eq.~\eqref{eq:minor_bound_derivativeB} is independent of $s$, letting $s\to0$ in \eqref{eq:entre2B} we obtain 
\begin{equation}
\nonumber
0> \int_{y_1}^{y_2} \frac{c_1(z)}{a(z)} \, dz\,,
\end{equation}
which is a contradiction. Thus assumption \eqref{eq:HA_x_xc} is absurd and therefore $X(t;0,0^+)\\ =\bar x = x_c(t)$. 

\smallskip

\textit{Step 4. Semigroup property in {\it \ref{it:s2}}.} Recall that $X(t;s,X(s;0,x))=X(t;0,x)$ for all $x>0$ and $s\in(0,t)$. By continuity we can pass to the limit $x \to 0$ so that $X(t;s,X(s;0,0^+))=X(t;0,0^+)$. Thanks to {\it Step 3} this reads $X(t;s,x_c(s))=x_c(t)$. 

\smallskip

\textit{Step 5. Proof that $X(t;s,0^+)<x_c(t)$ and the semigroup property in point  {\it \ref{it:s1}}}. 
Let $t\in(0,T)$, $s\in(0,t)$ and $x<x_0$. To start we point out that the limit $X(t,s,0^+)$ exists since $x\mapsto X(t;s,x)$ is positive and monotonically increasing.  By Lemma \ref{lem:invariance}, there holds that $X(s;0,x)>\min(x,x_0)=x$ for any $s\in(0,t)$. Thus, by Remark \ref{rem:no cross},
\[X(t;0,x)=X(t;s,X(s;0,x))\ge X(t;s,x)\,.\]
Taking $x\to 0^+$, we deduce that $x_c(t)=X(t;0,0^+)\geq X(t;s,0^+)$. Now we recall that $X(s;0,0^+)=x_c(s)>0$ by Lemma \ref{lem:separation by x_c(t)}. Thus, for   $y\in(0,x_c(s))$ we have, by Remark \ref{rem:no cross},
\[X(t;s,0^+)\leq X(t;s,y)\leq X(t;s,x_c(s))=x_c(t).\]
Then it follows that $x_c(t)> X(t;s,0^+)$ as desired: if we had $x_c(t)= X(t;s,0^+)$ for some $s>0$, we would deduce that $ X(t;s,y)= x_c(t)$ for any $y\in(0,x_c(s))$, which contradicts that $x\mapsto X(t;s,x)$ is a diffeomorphism onto $(0,\infty)$.	 Finally, since $X(t;\tau,X(\tau;s,x)) = X(t;s,x)$ for all $\tau \in(s,T)$, by continuity at the limit $x \to 0$ we obtain the semigroup property in point  {\it \ref{it:s1}}. 
\end{proof}

A useful consequence of Lemma \ref{prop:diffeo} is the next lemma,

\begin{lemma} \label{rem:twotimes}
For any $0<s_1<s_2<t$ we have $X(t;s_2,0^+)<X(t;s_1,0^+)$.
\end{lemma}

\begin{proof}
This can be shown by a contradiction argument; let us first assume that $X(t;s_2,0^+) > X(t;s_1,0^+)$. Since $x\mapsto X(t;s_1,x)$ is a diffeomorphism from $(0,\infty)$ to $(X(t;s_1,0^+),\infty)$, there exists $x>0$ such that $X(t;s_1,x)=X(t;s_2,0^+)$. Now we have that $X(s_2;t,X(t;s_2,0^+))=0$ thanks to semigroup properties, but we also have 
\[
X(s_2;t,X(t;s_2,0^+))= X(s_2;t,X(t;s_1,x))=X(s_2;s_1,x) > 0
\] by Lemma \ref{lem:invariance} -recall that $s_2>s_1$- and we reach a contradiction.\\ In the case  $X(t;s_2,0^+)=X(t;s_1,0^+)$ we conclude that $0=X(s_2;t,X(t;s_2,0^+))= X(s_2;s_1,0^+)>0$ by Lemma \ref{prop:diffeo}, a contradiction again. 
\end{proof}

Now we address the map given by $s\mapsto\sigma_t^{-1}(s)$.
The second part of Prop. \ref{prop:diffeomorphism_merge}, in the main text, is a particular consequence of the following result:
\begin{prop} \label{prop:diffeomorphism sigma_t} 
    For each $t\in(0,T)$, the map $s\mapsto\sigma_t^{-1}(s)$ is a decreasing $\mathcal{C}^1$-diffeomorphism from $(0,t)$ to $(0,x_c(t))$ satisfying, 
    \[\sigma_t^{-1}(s) \leq C(K,T) \]
    with $C(K,T)$ the constant in \eqref{eq:characteristics curves - uniform bounds}. Its derivative is given by 
    \begin{equation}  \label{prop:derivative of sigma-1}
     \dfrac{\partial \sigma_t^{-1}(s) }{\partial {s}} = - a(\sigma_t^{-1}(s)) w(s,0)\exp\left(-   \int_{s}^t \left(a\, \partial_x w\right)(\tau,\sigma_{\tau}^{-1}(s))\, d\tau \right)
    \end{equation} 
    for all $t\in(0,T)$ and $s\in(0,t)$. Moreover, $\sigma_t^{-1}(s) = \lim_{x\to 0^+} X(t;s,x)$ and $\sigma_\tau^{-1}(\sigma_t(x)) = X(\tau;t,x)$ for all $(t,x)\in\Omega_T^*$ and $\tau \in \Sigma_{t,x}$.
\end{prop}

We divide the proof of Proposition \ref{prop:diffeomorphism sigma_t} into several auxiliary lemmas.

\begin{lemma}\label{lem:decreasing sigma and limit}
  For all $t\in (0,T)$, $\sigma_t$ is  (strictly) decreasing on $(0,x_c(t))$ and
    \begin{equation}
    \label{eq:decreasing sigma and limit}
    \lim_{y\to 0^+}X(t,\sigma_t(x),y)=x \quad \mbox{for any}\ \ x\in(0,x_c(t)).
    \end{equation}
\end{lemma}

\begin{proof}
Let $\delta>0$ and $x_0$ given by \eqref{A6}. Fix $t\in(0,T)$ and $0<y<x<x_c(t)$. We know from Lemma \ref{lem:separation by x_c(t)} that $\sigma_t(y) \ge \sigma_t(x)$ and we will prove that equality cannot hold. Let us assume that $\sigma_t(x)=\sigma_t(y)$ and argue by contradiction. By Lemma \ref{lem:either-or property for sigma}, there exists $s_0\in(\sigma_t(x),t)$ such that $X(s_0;t,x)<x_0$ and by Remark \ref{rem:no cross} we also have $X(s_0;t,z)<x_0$ for all $z\in(y,x)$.
By Lemma \ref{lem:decomposition a phi}, 
\begin{multline}\label{eq:inter decreasing sigma 0}
    \left| \int_s^t \left(a\, \partial_x w\right)(r,X(r;t,z)) \, dr \right| \leq  \frac 1 \delta \int_0^t\int_0^{x_0} \left|\partial_x w(r,\xi)\right| \, d\xi dr \\+ \int_{s_0}^t \left|\left(a\, \partial_x w\right)(r,X(r;t,z))\right| \, dr
\end{multline}
for all $s\in (\sigma_t(x),s_0)$ and $z\in(y,x)$. Using Lemma \ref{lem:invariance} and the bound \eqref{eq:characteristics curves - uniform bounds} we can now estimate
\[
0 <x_m\coloneqq X(s_0;t,y) \le X(\tau;t,y) \le X(\tau;t,z) \le x_M\coloneqq C(K,T)(1+x)
\]
for all $z\in(y,x)$ and $\tau \in(s_0,t)$. We then have from Eq.~\eqref{eq:inter decreasing sigma 0},
\begin{equation}\label{eq:inter decreasing sigma 1}
    \left| \int_s^t \left(a\, \partial_x w\right)(r,X(r;t,z)) \, dr \right| \leq    \frac 1 \delta \int_0^t \int_0^{x_0} \left|\partial_\xi w(r,\xi)\right| \, d\xi dr + \left\|a\, \partial_x w \right\|_{L^\infty(\tilde \omega)}.
\end{equation}
with $\tilde \omega:=(0,T)\times (x_m,x_M)$. Finally, by Eqs \eqref{eq:derivatives of B} and \eqref{eq:inter decreasing sigma 1} we deduce that
\begin{multline*}
B(s;t,A(x))-B(s;t,A(y)) \\
\geq  \exp{\left(- \frac 1 \delta \int_0^t \int_0^{x_0} \left|\partial_\xi w(r,\xi)\right| \, d\xi dr - \left\|a\, \partial_x w \right\|_{L^\infty(\tilde \omega)}\right)} \int_y^x \frac{1}{a(z)} \, dz\,.
\end{multline*}
We may now take the limit $s\to \sigma_t(x)=\sigma_t(y)$ thanks to Lemma \ref{lem:either-or property for sigma} to deduce that  \smash{$\int_y^x \tfrac{1}{a(z)} \, dz\leq 0$}. This contradicts the strict positivity of $a$. Thus, $\sigma_t(y)>\sigma_t(x)$. 

We proceed now to the proof of the limit \eqref{eq:decreasing sigma and limit}. Let $x\in(0,x_c(t))$ be such that $\sigma_t(x)>0$. We remark that by Lemma \ref{prop:diffeo}, the limit $X(t;\sigma_t(x),0^+)$ exists. As a first step we prove that this limit is greater or equal than $x$. Let $\{x^n\}$ a positive decreasing sequence towards zero. Using Lemma \ref{lem:invariance}, $X(s;\sigma_t(x),x^n)\geq \min(x_n,x_0)>0$, for all $s\in(\sigma_t(x),T)$. Thus, $(\sigma_t(x),T) \subseteq \Sigma_{t,X(t;\sigma_t(x),x^n)}$. By definition of $\sigma_t$, we have
\[\sigma_t\left(X(t;\sigma_t(x),x^n)\right)\leq \sigma_t(x)\ \text{and then}\ X(t;\sigma_t(x),x^n)\geq x \]
(otherwise it contradicts the fact that $\sigma_t$ is decreasing). Moreover, by Remark \ref{rem:no cross}, $X(t;\sigma_t(x),x^n)$ is a decreasing sequence. Hence it converges to some $y\geq x$. 

We show now that $y=x$ arguing by contradiction. Assume that $y>x$. For $z\in(x,x_c(t)\wedge y)$ we have that $0<\sigma_t(z)<\sigma_t(x)$ by the first part of the proof. Then  Remark \ref{rem:no cross} yields, for all $s>\sigma_t(x)$ and $n\in \mathbb{N}$,
\[
X(s;t,x)<X(s;t,z)<X(s;t,y)\leq X(s;t,X(t;\sigma_t(x),x^n))=X(s;\sigma_t(x),x^n)\,.\]
Taking the limit $s\to\sigma_t(x)$ we obtain
\[0\leq X(\sigma_t(x);t,z)\leq x^n\,.\]
Taking now the limit $n\to \infty$ leads to  $X(\sigma_t(x);t,z)=0$, which contradicts that $0<\sigma_t(z)<\sigma_t(x)$ by Lemma \ref{lem:either-or property for sigma}. Then $y=x$, which concludes the proof.
\end{proof}

\begin{lemma} \label{lem:sigma homeo}
   For all $t\in (0,T)$, $\sigma_t$ is a (strictly) decreasing homeomorphism from $(0,x_c(t))$ to $(0,t)$.
\end{lemma}

\begin{proof}

 \textit{Step 1. Proof of $\sigma_t(x_c(t))=0$.} By definition, $\sigma_t(x_c(t))=\inf \Sigma_{t,x_c(t)}$, and $\Sigma_{t,x_c(t)}=\{s\geq0; X(s;t,x_c(t))>0\}$. By point {\it \ref{it:s2}} in Lemma \ref{prop:diffeo}, we have $X(s;t,x_c(t))=x_c(s)$. Thus, $\Sigma_{t,x_c(t)}=\{s\geq 0; x_c(s)>0\}$. Thanks to Lemma \ref{lem:separation by x_c(t)}, $x_c(s)$ is positive for any $s>0$ and hence $\sigma_t(x_c(t))=\inf \Sigma_{t,x_c(t)}=0$.
 
 In the next two steps, we characterize the sequential continuity of $\sigma_t$ both from the left and from the right. There is no loss of generality in restricting ourselves to monotone  sequences.
 
 \textit{Step 2. The map $x\mapsto \sigma_t(x)$ is left continuous.} Let $x\in(0,x_c(t)]$ and let $\{x^n\} \subset (0,x_c(t))$ be an increasing sequence converging to $x$ with $x^n<x$ for all $n\geq 1$. By Lemma \ref{lem:decreasing sigma and limit}, $\sigma_t(x^n) > \sigma_t(x)$ for all $n\geq 1$ and $\{\sigma_t(x^n)\}$ is a decreasing sequence, thus it converges. To show the sequential continuity of $\sigma_t$ we argue by contradiction; therefore we 
 assume that $\sigma_t(x^n)$ converges to $s >\sigma_t(x)$. Note that in particular $\sigma_t(x^n)\geq  s$. Assumption \eqref{A6} provides us with  $\delta>0$ and $x_0$ such that $w(t,y)\ge \delta $ for all $y\in(0,x_0)$, so that $V(t,A(y))
 \geq \delta$ 
 and then
\[\frac{\partial B(t;s,A(y))  }{\partial s}\leq -\delta I(t;s,A(y))
 \leq 0\,, \quad \mbox{for}\, s\in(\sigma_t(x),\sigma_t(x^n)).
\] 
 Thus, as $A^{-1}$ is increasing, we have that 
\[ X(t;\sigma_t(x^n),y)\leq X(t;s,y)\leq X(t;\sigma_t(x),y)
\quad \mbox{for all}\, y\in(0,x_0).
\]
Using Lemma \ref{lem:decreasing sigma and limit} we may take the limit $y\to 0$ to obtain
\[ x^n \leq X(t;s,0^+)\leq x.\]
Taking next the limit $n\to\infty$ we deduce that $X(t;s,0^+)=x=X(t;\sigma_t(x),0^+)$, which contradicts Lemma \ref{rem:twotimes} since we had assumed that  $s>\sigma_t(x)$. Therefore $\sigma_t(x^n)$ converges to $\sigma_t(x)$ as desired.

\smallskip
 
\textit{Step 3. The map $x\mapsto \sigma_t(x)$ is right continuous.} Let $x\in(0,x_c(t))$ and take $\{x^n\}$ a decreasing sequence converging to $x$ and such that $x<x^n< x_c(t)$. Thus $\{\sigma_t(x^n)\}$ is increasing and $\sigma_t(x^n)< \sigma_t(x)$. We show again the sequential continuity by means of a contradiction argument. Assume that $\sigma_t(x^n)\to s < \sigma_t(x)$ and hence $\sigma_t(x^n) < s$.
Similarly to Step 2, for any $y<x_0$ we have
\[ X(t;\sigma_t(x^n),y)\geq X(t;s,y)>X(t;\sigma_t(x),y)\,. \]
Using Lemma \ref{lem:decreasing sigma and limit} and taking the limit $y\to 0$ we obtain
\[ x^n \geq X(t;s,0^+)\geq x\,.\]
 Now we deduce that $X(t;s,0^+)=x$ taking the limit $n\to\infty$. But for $u\in(s,\sigma_t(x))$ we have $X(u;s,0^+)>0$ and also
 \[X(u;s,0^+)=X(u;t,X(t;s,0^+))=X(u;t,x)>0
 \] by Lemma \ref{prop:diffeo}, which contradicts that $u<\sigma_t(x)$. This shows that $\sigma_t(x^n)\to \sigma_t(x)$ as $n\to \infty$.

\smallskip

\textit{Step 4. Proof of $\lim_{x\to 0} \sigma_t(x)=t$.} Let $\{x^n\}$ be a positive decreasing sequence converging to zero. Then $\sigma_t(x^n)\leq t$ is  increasing and converges to $\bar s \in(0,t]$. We use a contradiction argument to prove that $\bar s=t$. Assume that $\bar s<t$. There exists $N\geq 1$ such that $x^n < \min(x_0,x_c(t))$ for all $n\geq N$. By Lemma \ref{lem:invariance}, \[
A(x^n) - B(s;t,A(x^n)) \geq \delta (t - s) \quad \mbox{for all}\, s\in(\sigma_t(x^n),t) \, \mbox{and}\, n\geq N.  
\]
Then, for  $\eta \in (0, t - \bar s)$ we have that $B(t-\eta;t,A(x^n))$ is well defined for all $n\geq N$ and
\[ B(t-\eta;t,A(x^n)) \leq A(x^n) -  \delta \eta\,. \]
We can choose $n$ large enough such that $A(x^n) - \eta \delta<0$, since $A(x_n)$ converges to zero. This contradicts that $B(t-\eta;t,A(x^n))>0$ by construction. Thus we must have $\bar s = t$.

\smallskip

\textit{Step 5. Conclusion.} Putting together what we have proved so far, $\sigma_t$ is a strictly decreasing, positive and continuous function on $(0,x_c(t))$. Therefore it is an homeomorphism onto its image $(0,t)$ since $\lim_{x\to 0} \sigma_t(x)=t$, and, by continuity, $\lim_{x\to x_c(t)} \sigma_t(x) = 0$.  

\end{proof}

\begin{lemma} \label{lem:bound sigma}There holds that
\[
\sigma_t^{-1}(s) = \lim_{x\to 0} X(t;s,x) \leq C(K,T)\quad \mbox{for all}\ t\in(0,T)\ \ \mbox{and}\ s\in(0,t)
\]
with $C(K,T)$ the constant in \eqref{eq:characteristics curves - uniform bounds}.
\end{lemma}
 
 \begin{proof}
 Let $t\in(0,T)$ and $s\in(0,t)$. Consider  a positive decreasing sequence $\{x^n\}$  converging to zero; note that $\{X(t;s,x^n)\}$ is a decreasing sequence. Thus, by Lemma \ref{lem:sigma homeo}, $\{\sigma_t(X(t;s,x^n))\}$ is increasing and verifies that $\sigma_t(X(t;s,x^n))\leq s$; hence $\sigma_t^{-1}(s) \leq  X(t;s,x^n)$. 
 Since the sequence $\{X(t;s,x^n)\}$ decreases, it converges to some $\bar x\geq\sigma_t^{-1}(s)$ and in particular $X(t;s,x^n)\geq \bar x$ for all $n\geq 1$. We argue by contradiction to show that equality holds. Assume $\bar x>\sigma_t^{-1}(s)$ and let $y\in(\sigma_t^{-1}(s),\bar x)$, hence  $\sigma_t(y)< s$ and $X(s;t,y)>0$. But $y < \bar x \leq X(t;s,x^n)$ and then $X(s;t,y) < x^n \to 0$ as $n\to \infty$, which contradicts that $X(s;t,y)>0$. Thus $\bar x = \sigma_t^{-1}(s)$ as claimed. The bound is obtain by taking the limit $x\to 0$ in  \eqref{eq:characteristics curves - uniform bounds}. 
 \end{proof}

\begin{lemma} 
    For all $t\in(0,T)$, $\sigma_t^{-1}$ is a decreasing $C^1$-diffeomorphism from $(0,t)$ to $(0,x_c(t))$ and its derivative is given by  Eq.~\eqref{prop:derivative of sigma-1}.
\end{lemma}
\begin{proof}
For any $t\in(0,T)$, $s\in(0,t)$ and $x>0$ we have from Eq.~\eqref{eq:derivatives of B} that
\begin{multline}\label{eq:to intervert}
B(t;t,A(x)) - B(t;s,A(x)) =  \int_s^t \dfrac{\partial B(t;\tau,A(x))}{\partial \tau} \, d\tau\\
=-  \int_s^t w(\tau,x)\, \exp\left(-\int_\tau^t \left(a\, \partial_x w \right)(r,X(r;\tau,x))\, dr \right) \, d\tau\,.
\end{multline}
We want to take the limit $x \to 0^+$ in this equation, since as $x\to 0^+$ we have $B(t,t,A(x))=A(x) \to 0$ and also $B(t;s,A(x))=A(X(t;s,x))\to A(\sigma_t^{-1}(s))$ from Lemma \ref{lem:bound sigma}, as $x\to 0^+$. Interchanging the limit and the integral we get a formula for $\sigma_t^{-1}(s)$ that will enable us to compute the derivatives. We split this argument in three steps.

\smallskip

\textit{Step 1.} We justify that for all $t\in(0,T)$ and $\tau\in(0,t)$,
\begin{multline}\label{eq:limit_1}
\lim_{x\to 0^+} \dfrac{\partial B(t;\tau,A(x))}{\partial \tau}\\
= -\lim_{x\to 0^+} w(\tau,x)\, \exp\left(-\int_\tau^t \left(a\, \partial_x w\right)(r,X(r;\tau,x))\, dr \right)
 \\ 
=
-w(\tau,0) \exp\left(- \displaystyle  \int_\tau^t \left(a\, \partial_x w\right)(r,\sigma_r^{-1}(\tau))\, dr\right).
\end{multline}
Let $t\in (0,T)$, $\tau\in(0,t)$. To prove the limit in \eqref{eq:limit_1}, we will split the integral above in two parts,
\begin{multline}\label{eq:limit_int_2}
\int_\tau^t \left(a\, \partial_x w\right)(r,X(r;\tau,x))\, dr = \int_\tau^{\tau_0} \left(a\, \partial_x w\right)(r,X(r;\tau,x))\, dr \\+ \int_{\tau_0}^t \left(a\, \partial w\right)(r,X(r;\tau,x))\, dr\,.
\end{multline}
Here $\tau_0$ is chosen small enough and independent of $x$, so that it allows to make the change of variable $r\mapsto z=X(r;\tau,x)$ (as in the proof of Lemma \ref{lem:decomposition a phi}) on $(\tau,\tau_0)$ and to be away from zero on $(\tau_0,t)$.

Let $\veps>0$ and consider $\delta>0$ and $x_0>0$ given by \eqref{A6}. We recall that $\partial_x w$ is integrable by assumption \eqref{A5}, therefore, we can find  $x_1\in(0,x_0)$ such that $\int_0^T \int_0^{x_1}|\partial_y w(r,y)|\, dydr <\delta \veps/2$. Let now $x_2\in(0,x_1)$. There exists $\tau_0\in(\tau,t)$ such that $X(\tau_0;\tau,x_2)<x_1$ by continuity in the first variable. Hence $X(\tau_0;\tau,x)<X(\tau_0;\tau,x_2)<x_1$ for all $x\in(0,x_2)$ by Remark \ref{rem:no cross}. Using Lemma \ref{lem:decomposition a phi} with a similar reasoning as in \eqref{eq:210}  we have that
\begin{equation} \label{eq:bound-aphi-0}
    0\leq \int_{\tau}^{\tau_0} \left|\left(a\, \partial_x w\right)(r,X(r;\tau,x))\right|  \, dr \leq \frac{1}{\delta}\int_0^t\int_0^{x_1}\left|\partial_y w(r,y)\right|\, dydr
    \leq \frac \veps 2
\end{equation} 
for all $x\in(0,x_2)$. Also, by Fatou's lemma, as $x \to 0$,
\begin{equation} \label{eq:bound-aphi-0-limit}
 \int_{\tau}^{\tau_0} \left|\left(a\, \partial_x w\right)(r,\sigma_r^{-1}(\tau))\right|\, dr \leq \frac\veps 2.
\end{equation}

For the second term in Eq.~\eqref{eq:limit_int_2}, we bound the integrand uniformly in $x$. Indeed, by the bound \eqref{eq:characteristics curves - uniform bounds} that is $C(K,T)>0$ independent of $\tau$, we have $X(r;\tau,x)\leq x_M\coloneqq C(K,T)(1+x_0)$ for all $r\in(\tau,T)$ and $x\in(0,x_2)$. We also have 
\begin{equation}\label{eq:bound below B}
 B(\tau_0;\tau,A(x)) \geq \delta (\tau_0-\tau) + A(x) \geq \delta (\tau_0-\tau) 
\end{equation}
by Lemma \ref{lem:invariance}. Thus, there holds that 
\[
x_0> X(\tau_0;\tau,x)\geq x_m {:= A^{-1}(\delta(\tau_0-\tau))
}>0\quad \mbox{for all}\, x\in(0,x_2),
\]
 which entails, by Lemma \ref{lem:invariance}, that 
 \[X(r;\tau,x)=X(r;\tau_0,X(\tau_0;\tau,x))\ge X(\tau_0;\tau,x)\ge x_m\]
  for all $r \ge \tau_0$ and for every $x\in (0,x_2)$. Hence, the map $r\mapsto(a\,\partial_x w)(r,X(r;\tau,x))$ is uniformly bounded in $r\in(\tau_0,t)$ and $x\in(0,x_2)$, which justifies that 
\begin{equation} \label{eq:bound-aphi-1}
    \lim_{x\to 0^+} \int_{\tau_0}^t \left(a \, \partial_x w\right)(r,X(r;\tau,x))  \, dr = \int_{\tau_0}^t \left(a \, \partial_x w\right)(r, \sigma_r^{-1}(\tau)) \, dr.
\end{equation}
Finally, by Eqs. \eqref{eq:bound-aphi-1}, \eqref{eq:bound-aphi-0} and \eqref{eq:bound-aphi-0-limit} we deduce that for all $\veps>0$,
\[ \lim_{x\to 0^+} \left| \int_{\tau}^{t} \left(\left(a\, \partial_x w\right)(r,X(r;\tau,x))  -   \left(a\, \partial_x w\right)(r,\sigma_r^{-1}(\tau)) \right) \, dr  \right| \leq \veps  \]
and in this way we obtain 
\eqref{eq:limit_1}.

\smallskip

\textit{Step 2.} Now we bound the derivative of $B$ in order to pass to the limit $x\to 0$ in Eq.~\eqref{eq:to intervert}.  Let $t\in (0,T)$. We split the integral in \eqref{eq:to intervert} in three parts, again with the idea of separating the contributions where $X(r;\tau,x)$ is close to zero and away from it:
\begin{multline*}
 \int_s^t w(\tau,x)\exp\left(-\int_\tau^t \left(a\, \partial_x w\right)(r,X(r;\tau,x))\, dr\right) \, d\tau
 \\ = \int_{t_0}^t w(\tau,x)\exp\left(-\int_\tau^t \left(a\, \partial_x w\right)(r,X(r;\tau,x))\, dr\right) \, d\tau\\
 + \int_s^{t_0} w(\tau,x)\exp\Big(-\int_\tau^{\tau+\tau_0} \left(a\, \partial_x w\right)(r,X(r;\tau,x))\, dr\\
 -\int_{\tau+\tau_0}^t \left(a\, \partial_x w\right)(r,X(r;\tau,x))\, dr\Big) \, d\tau \,.
\end{multline*}
This holds for some $0<t_0<t_0+\tau_0<t$, with $t_0$ sufficiently close to $t$ and $\tau_0$ small enough, both independent from $x$, as we explain in what follows. Consider again $\delta>0$ and $x_0>0$ given by \eqref{A6} and let $x_1\in(0,x_0)$. As $X(t;t,x_1)=x_1$,  by continuity, there exists $t_0<t$ such that $X(t;\tau,x_1)<x_0$ for all $\tau\in(t_0,t)$ and hence 
\[
X(r;\tau,x)<X(t;\tau,x)<X(t;\tau,x_1)<x_0\,, \ \text{for all}\ r\in(\tau,t),\, x\in(0,x_1),\, \tau\in(t_0,t)\,.
\] 
Using Lemma \ref{lem:decomposition a phi}, we get that  
\[  \left| \int_{\tau}^{t}  \left(a\, \partial_x w\right)(r,X(r;\tau,x)) \, dr \right| \leq \frac 1 \delta \int_0^t \int_0^{x_0} \left|\partial_y w(r,y)\right|\, dy dr  \]
for all $\tau\in(t_0,t)$ and $x\in(0,x_1)$.

Let now $\tau \in (0,t_0)$. Since $x_1<x_0$ there exists $0<\tau_0 < t-t_0$ such that $A(x_1)+\rho\tau_0<A(x_0)$. Hence by Eq.~\eqref{eq:edo-B}, 
\[
B(\tau+\tau_0;\tau,A(x)) \leq A(x) + \rho \tau_0 \leq A(x_0)\quad \mbox{for all}\, x\in(0,x_1) 
\]
and by Lemma \ref{lem:invariance}, 
\[X(r;\tau,x)<X(\tau+\tau_0;\tau,x)<x_0\, \quad \mbox{for all}\, r\in(\tau,\tau+\tau_0).\] 
Note that the condition on $\tau_0$ ensures $\tau+\tau_0<t$ whenever $\tau\in(0,t_0)$. On the other hand we have 
\[
B(\tau+\tau_0;\tau,A(x)) \geq A(x) + \delta \tau_0 \geq \delta \tau_0
\]
 and, by Lemma \ref{lem:invariance}, $X(r;\tau,x) \ge x_m\coloneqq A^{-1}(\delta \tau_0)$ for all $r\in(\tau+\tau_0;t)$ and $x\in(0,x_1)$. Since $X(r;\tau,x) \leq x_M\coloneqq C(K,T)(1+x_1)$ for all $\tau\in(0,T)$, $r\in(\tau,T)$ and $x\in(0,x_1)$, we let $\tilde \omega:=(0,T)\times (x_m,x_M))$ and then we have 
\[\left| \int_{\tau+\tau_0}^t \left(a\, \partial_x w\right)(r,X(r;\tau,x)) \, dr \right| \leq T \left\|a\, \partial_x w\right\|_{L^\infty (\tilde \omega)} \]
for all $\tau \in(0,t_0)$ and $x\in(0,x_1)$. Note that $x_m$ does not depend here on $\tau\in(0,t_0)$, contrary to the construction from Eq.~\eqref{eq:bound below B}. Finally, by Lemma \ref{lem:decomposition a phi},
\[\left| \int_{\tau}^{\tau+\tau_0} \left(a\, \partial_x w\right)(r,X(r;\tau,x))  \, dr \right|  \leq \frac 1 \delta \int_0^t \int_0^{x_0} \left|\partial_y w(r,y)\right| \, dydr.\]
Combining these results we obtain that $\dfrac{\partial B(t;\tau,A(x))}{\partial \tau}$ is uniformly bounded in $\tau \in(0,t)$ and $x\in(0,x_1)$, namely
\begin{equation}
\nonumber
    \left|\dfrac{\partial B(t;\tau,A(x))}{\partial \tau} \right| \leq \|w\|_{L^\infty(\tilde \omega)}   \exp\left(\frac{2}{\delta}\|\partial_x w\|_{L^1((0,T)\times(0,x_0))}  + T \left\|a\,\partial_x w\right\|_{L^\infty (\tilde \omega)}  \right)\:.
\end{equation} 

\smallskip

\textit{Step 3.} Now we pass to the limit $x\to 0$ in Eq.~\eqref{eq:to intervert}, where the interchange of limits and integrals is justified and we obtain
\[ \sigma_t^{-1}(s) = A^{-1}\left(\int_s^t w(\tau,0) \exp\left(-\int_\tau^t \left(a\, \partial_x w\right)(r,\sigma_r^{-1}(\tau))\, dr\right) \, d\tau \right)\]
for all $t\in(0,T)$ and $s\in(0,t)$. Clearly the right-hand side is continuously differentiable since $A^{-1}$ is and we easily identify the derivative of $\sigma^{-1}_t(s)$.

\end{proof}

\subsection{Representation formula and regularity properties}
\label{Ann:5}

Once we have the tools introduced in the previous subsection we can proceed to prove the statements of Theorem \ref{th:compilation}. Thanks to Lemma \ref{prop:diffeo} and Proposition \ref{prop:diffeomorphism sigma_t}, we define, for a.e. $(t,x)\in \Omega_T^*$, 
  \begin{equation}
\label{eq:mild solution}
f(t,x) = f^{\rm in}(X(0;t,x)) J(0;t,x) \mathbf 1_{(x_c(t),\infty)}(x) + G(\sigma_t(x))|\sigma_t'(x)|\mathbf 1_{(0,x_c(t))}(x)\,.
\end{equation}
This ensures that $f$ solves \eqref{eq:linearproblem}. We now provide several intermediate statements that serve as a proof of the remaining points in Theorem \ref{th:compilation}. Recall that $f^{\rm in}$ satisfies \eqref{H6}.

\begin{lemma}\label{lem:compactness}
	The family $\set{f(t,\cdot)}{t\in (0,T)}$ constructed via Eq.~\eqref{eq:mild solution} is weakly relatively compact in $L^1((0,\infty),(1+x)dx)$. 
	In particular, 
	\begin{equation} \label{eq:bound-1+x-f}
	\sup_{t\in[0,T)}\int_0^\infty (1+x)f(t,x)\, dx < \infty.
	\end{equation}
\end{lemma}

\begin{proof}
	 The result will follow as a consequence of Dunford-Pettis' theorem, see  \emph{e.g.} \cite[Chap. IV.8]{Dunford} . Since $f$ is nonnegative, we are to prove the following:
	\begin{enumerate} [align=left, leftmargin=2pt, labelindent=\parindent,label=\textup{(\arabic*)},itemsep=3pt]
	\item Bound \eqref{eq:bound-1+x-f},
	\item $\lim\limits_{n \to +\infty} \sup\limits_{t\in[0,T]}\displaystyle  \int_n^\infty f(t,x) (1+x)\,dx = 0$,
	\item For all $\veps >0$, there exists $\delta >0$ such that
	\[ \sup_{t\in[0,T]} \int_E f(t,x) (1+x)\,dx <\veps\]
	for every Lebesgue measurable set $E$ with measure $|E|<\delta$. 
	\end{enumerate}
	
	Point 1. We integrate Eq.~\eqref{eq:mild solution} and use the diffeomorphisms in Lemma \ref{prop:diffeo} and Proposition \ref{prop:diffeomorphism sigma_t} to obtain 
	\begin{equation*}
	\int_0^\infty f(t,x) \, dx = \int_0^\infty f^{\rm in}(x)\,  dx + \int_0^t G(s) \, ds,
	\end{equation*}
	for each $t\in(0,T)$. This is uniformly bounded since $G$ is bounded and $f^{\rm in}$ belongs to $L^1((0,\infty),(1+x)dx)$. 
	In a similar way, using the bound~\eqref{eq:characteristics curves - uniform bounds} and the bound in Proposition \ref{prop:diffeomorphism sigma_t} we have that for each $t\in(0,T)$,
	\begin{equation*} 
	\int_0^\infty x f(t,x)\,dx = \int_0^\infty X(t;0,x) f^{\rm in}(x)\,dx + \int_0^t \sigma_t^{-1}(s) G(s)
	\,ds
	\end{equation*}
	is uniformly bounded.  This proves \eqref{eq:bound-1+x-f}. 
	
	Point 2. Note first that there exists a constant $C(K,T)>0$ such that   $x_c(t) \leq C(K,T)$ for all $t\in(0,T)$; this follows from bound \eqref{eq:characteristics curves - uniform bounds} and Lemma \ref{prop:diffeo}. Choose $N$ large enough such that $N\geq x_c(t)$ for all $t\in(0,T)$. Then, integrating \eqref{eq:mild solution} and changing variables we obtain that
	\[\int_n^\infty (1+x)f(t,x)\, dx = \int_{X(0;t,n)}^\infty (1+X(t,0,x)) f^{\rm in}(x) \, dx\]
	for all $t\in(0,T)$ and  $n\ge N$. Since $X(t,0,x) \leq C(K,T)(1+x)$ again from bound \eqref{eq:characteristics curves - uniform bounds}, we have that
	\[\int_n^\infty (1+x)f(t,x)\, dx \le C(K,T) \int_{X(0;t,n)}^\infty (1+x) f^{\rm in}(x)\,  dx\]
	increasing the value of the constant $C(K,T)$ if needed. Next we notice that $n \le C(K,T)(1+X(0;t,n))$ after \eqref{eq:characteristics curves - uniform bounds} and the semigroup property. Hence, thanks to integrability of $f^{\rm in}$, we can pass to the limit $n\to \infty$, uniformly in $t$, to obtain the desired property.
	
	Point 3. Let $E$ be a Lebesgue measurable set. We estimate the integrals over $E\cap(0,x_c(t))$ and $E\cap(x_c(t),\infty)$ separately. Thanks to Eqs.~\eqref{eq:characteristics curves - derivatives in x and t} and \eqref{eq:characteristics curves - uniform bounds}, we have
	\begin{multline}\label{eq:concentration-inter--1}
	\int_{E\cap(x_c(t),\infty)}(1+x)f(t,x)\,dx = \int_{E\cap(x_c(t),\infty)} (1+x) f^{\rm in}(X(0;t,x))J(0;t,x) \, dx \\
	\leq C(T) \int_{X(0;t,E\cap(x_c(t),\infty))} (1+x) f^{\rm in}(x)\, dx
	\end{multline}
	for some constant $C(T)>0$ independent of time $t\in[0,T)$. Let $x_0$ be given by \eqref{A6} and let $\bar x$ be such that $X(s;0,x_0)\leq \bar x$ for all $s\in[0,T)$ which is possible thanks to  Eq. \eqref{eq:characteristics curves - uniform bounds}. Note that for all $s,t\in[0,T)$ and $x>\bar x$ we have 
	\[	X(s;t,x)>X(s;t,\bar x)>X(s;t,X(t;0,x_0))=X(s;0,x_0)>x_0\]
	This is due to the monotonicity (Remark \ref{rem:no cross}), the semigroup property and invariance. Now we estimate the measure of $X(0;t,E\cap(x_c(t),\infty))$ for $t\in[0,T)$ as follows:
	\begin{multline}\label{eq:concentration-inter-0}
	|X(0;t,E\cap(x_c(t),\infty))| = \int_{E\cap(x_c(t),\infty)} J(0;t,x)\, dx \\
	\leq \int_{E\cap(x_c(t),\bar x)}  J(0;t,x)\, dx + |E\cap(\bar x,\infty)|\exp\left(\left\|\partial_x v\right\|_{L^\infty((0,T)\times (x_0,\infty))} \right)\,. 
\end{multline}
	Here we used Eq. \eqref{eq:characteristics curves - derivatives in x and t} and \eqref{A1p}. Next we proceed to bound the Jacobian. Since $A(X(0;t,x))=B(0;t,A(x))$ using the derivatives in the third variable for $X$ and $B$ we get 
	\[J(0;t,x) = \frac{a(X(0;t,x))}{a(x)} I(0;t,A(x)) \]
    for all $x>x_c(t)$, with $I$ in Eq. \eqref{eq:derivatives of B}. To proceed further we use Eq. \eqref{eq:characteristics curves - uniform bounds} to fix $x^*$ such that $X(0;t,x) \leq x^*$ for all $x\in(x_c(t),\bar x)$. By Eq. \eqref{eq:concentration-inter-0} above, we obtain
    \begin{equation} \label{eq:concentration-inter-1}
    |X(0;t,E\cap(x_c(t),\infty))| \leq \|a\|_{L^\infty(0,x^*)}\! \int_{E\!\cap\!(x_c(t),\bar x)} \!\frac{I(0;t,A(x))}{a(x)} \, dx + C(T)|E|\,.
    \end{equation}
    We now bound $I$. Given $x\in(x_c(t),\bar x)$, we either have $X(0;t,x)<x_0$  or $X(0;t,x)\geq x_0$; we discuss both cases in turn. On one hand, if $X(0;t,x) \geq x_0$ we use Lemma \ref{lem:invariance} to deduce that for all $s\in(0,T)$, $x_0 \le X(s;t,x)\le x^*$. Thus, noticing that $a \partial_x w$ is locally bounded on $\Omega_T$, because of \eqref{A2} and \eqref{A3},
    \begin{equation}\label{eq:concentration-inter-2}
    I(0;t,A(x)) \leq \exp\left(
    \left\|a \partial_x w \right\|_{L^\infty((0,T)\times (x_0,x^*))}
    \right)\,.
    \end{equation}
    On the other hand, if $X(0;t,x)<x_0$, there exists $s_0$ such that $X(s_0;t,x)=x_0$ and then $X(s;t,x)<x_0$ for all $s\in(0,s_0)$. So, by Lemma \ref{lem:decomposition a phi} 
    \begin{equation} \label{eq:concentration-inter-3}
    |I(0;t,A(x))| \leq \exp \left(
    \tfrac 1 \delta \int_0^T\int_0^{x_0} 
\partial_y w(r,y)\, dydr
    +\left\|a \partial_x w\right\|_{L^\infty((0,T)\times (x_0,x^*))}
      \right)\:. 
    \end{equation}
     In conclusion, combining Eqs. \eqref{eq:concentration-inter-1}, \eqref{eq:concentration-inter-2} and \eqref{eq:concentration-inter-3} 
    we obtain 
    \[|X(0;t,E\cap(x_c(t),\infty))| \leq C(T) \left(\int_{E\cap(0,\bar x)} \frac{1}{a(x)}\, dx + |E|\right)\,. \]
    Given that $1/a\in L^1(0,1)$ and $f^{\rm in}(x)$ is integrable, Eq.~\eqref{eq:concentration-inter--1} entails 
	\begin{equation} \label{eq:limit E 1}
	\lim_{|E| \to 0} \sup_{t\in[0,T]}\int_{E\cap(x_c(t),\infty)}(1+x)f(t,x)\, dx = 0\,.
	\end{equation}
	
	It remains to do the same with
	\[ \int_{E\cap(0,x_c(t))} (1+x)f(t,x) \, dx = \int_{E\cap (0,x_c(t))} (1+x)  G(\sigma_t(x)) |\sigma_t'(x)|\, dx\,. \]
	Recall that $G(t)$ and $x_c(t)$ are uniformly bounded on $(0,T)$. Therefore, there is some $C(T)$ such that 
	\begin{equation}\label{eq:concentre-0}
	\int_{E\cap(0,x_c(t))}(1+x) f(t,x)\, dx \leq C(T) \int_{E\cap (0,x_c(t))} |\sigma_t'(x)|\, dx. 
	\end{equation}
	We now consider this last integral. We observe that for all $x\in(0,x_c(t))$ 
	\begin{equation}
	\label{eq:deriv stigmatx}
	\sigma_t'(x) = \frac{1}{\sigma_t^{-1}{}'(\sigma_t(x))} = - \frac{1}{a(x)w(\sigma_t(x),0)} \exp\left( \int_{\sigma_t(x)}^t \left(a\, {\partial_x w}\right)(\tau,X(\tau;t,x))\, d\tau \right)\,, 
	\end{equation}
	where we used that $X(\tau;t,x) = \sigma_\tau^{-1}(\sigma_t(x))$ by Proposition~\ref{prop:diffeomorphism sigma_t}. Thanks to Lemma \ref{lem:either-or property for sigma}  we have that $\lim_{s\to \sigma_t(x)} X(s;t,x)=0$ whenever $x\in(0,x_c(t))$. Thus, for each $(t,x)\in(0,T)\times(0,x_c(t))$, there exists $s_0\in(\sigma_t(x),t]$ such that $X(s;t,x)<x_0$ for all $s\in(\sigma_t(x),s_0)$ where $x_0$ is given by \eqref{A6}. Using Lemma \ref{lem:decomposition a phi}, 
	\[ \int_{\sigma_t(x)}^t\!  \left(a\, \partial_x w\right)\!(\tau,X(\tau;t,x))\, d\tau \! \leq \! \frac{1}{\delta}\! \int_0^T\! \!\int_0^{x_0}\!
	|\partial_y(r,y)|\, dydr + \left\|a\, \partial_x w\right\|_{L^\infty((0,T)\times (x_0,C(T))}.\]
	 Here $C(T)>0$ is some constant which bounds $X(s;t,x)$ uniformly in $s,t\in (0,T)$ and $x\in(0,x_c(t))$ -see Lemma \ref{lem:ode-X}. Finally, thanks to \eqref{A6}, we have 
	\[|\sigma_t'(x)| \leq   \frac{C(T)}{\delta a(x)}  \]
    for all $s\in(0,T)$, again by Lemma \ref{lem:invariance}. The right-hand side of this last estimate is integrable around the origin by \eqref{A4} and hence, by  Eq.~\eqref{eq:concentre-0},
	\begin{equation} \label{eq:limit E 2}
	\lim_{|E| \to 0} \int_{E\cap(0,x_c(t))}(1+x)f(t,x)\,dx \leq  \frac{C(T)}{\delta} \lim_{|E|\to0} \int_{E\cap(0,x_c(t))}\frac{1}{a(x)}\, dx = 0.  
	\end{equation}
	Combining limits \eqref{eq:limit E 1} and \eqref{eq:limit E 2} finishes the proof.
\end{proof}

\begin{lemma}
	The function $f$ in Eq.~\eqref{eq:mild solution} satisfies 
	\begin{multline} \label{eq:ww}
	\int_0^T \int_0^\infty (\partial_t\vphi(t,x) + v(t,x)
	\partial_x\vphi(t,x))f(t,x)\,dx\, dt \\
	+ \int_0^\infty \vphi(0,x)f^{\rm in}(x) \, dx + \int_0^T \vphi(t,0) G(t)
	\, dt = 0
	\end{multline}
	for all $\vphi\in\Cc^1_c([0,T)\times[0,\infty))$.
\end{lemma}

\begin{proof}
	Let $\vphi\in\Cc^1_c([0,T)\times[0,+\infty))$, and define
	\begin{equation}\label{eq:def-psi}
	\psi(t,x) = - ( \partial_t\vphi(t,x) + v(t,x)
	)\partial_x\vphi(t,x)),  \quad    (t,x)\in \Omega_T\,.
	\end{equation}
	Using the definition of $f$ in Eq.~\eqref{eq:mild solution}, its integrability in Lemma~\ref{lem:compactness} and equation~\eqref{eq:def-psi}, we obtain 
	\begin{multline}\label{eq:intermediate-mild-1}
	\int_0^T \int_0^\infty (\partial_t\vphi(t,x) + v(t,x)
	)\partial_x\vphi(t,x))f(t,x)\, dx\, dt \\
	= - \int_0^T \int_{x_c(t)}^\infty \psi(t,x) f^{\rm in}(X(0;t,x))J(0;t,x)\, dx\, dt \\ -    \int_0^T \int_0^{x_c(t)} \psi(t,x) G(\sigma_t(x)) |\sigma_t'(x)|\, dx\, dt\,.
	\end{multline} 
	Using the changes of variables in Lemma \ref{prop:diffeo} and Proposition~\ref{prop:diffeomorphism sigma_t} with Fubini's theorem, we have
	\begin{multline} \label{eq:weak-etape1}
	\int_0^T \int_{x_c(t)}^\infty \psi(t,x) f^{\rm in}(X(0;t,x))J(0;t,x)\, dx\, dt \\
	=  \int_0^\infty \left(\int_0^T \psi(t,X(t;0,x))\right) f^{\rm in}(x)\, dx\, dt\,,
	\end{multline}
	and
	\begin{multline} \label{eq:weak-etape2}
	\int_0^T \int_0^{x_c(t)} \psi(t,x) G( u(\sigma_t(x)) ) |\sigma_t(x)'|\, dx\, dt  \\
	=  \int_0^T \left(\int_s^T \psi(t,\sigma_t^{-1}(s))\, dt \right)G(s)\, ds\,.
	\end{multline}
	By the definition of the characteristics curves \eqref{eq:characteristics curves} and using the definition of $\psi$ in Eq.~\eqref{eq:def-psi}, we have
	\begin{equation} \label{eq:deriv-phi-x} 
	\frac{\partial }{\partial s} [\vphi(s,X(s;t,x))] = - \psi(s,X(s;t,x))
	\end{equation}
	for all $(t,x)\in \Omega_T$ and $s\in(\sigma_t(x),T)$. We stress that this equation remains true for $t=0$ since, by Lemma  \ref{lem:either-or property for sigma}, $X(s;0,x)>0$ for all $s>0$. Hence, integrating Eq.~\eqref{eq:deriv-phi-x} over $(0,T)$ and since $\vphi(T,x)=0$ for all $x>0$, this yields 
	\begin{equation*}
	\vphi(0,x) = \int_0^T \psi(t;X(t;0,x))\, dt
	\end{equation*}
	for $x>0$. We can insert this relation into equation~\eqref{eq:weak-etape1} to obtain
	\begin{equation} \label{eq:weak-etape1-bis}
	\int_0^T \int_{x_c(t)}^\infty \psi(t,x) f^{\rm in}(X(0;t,x))J(0;t,x)\, dx\, dt =  \int_0^\infty \vphi(0,x) f^{\rm in}(x)\, dx\,  dt\,.
	\end{equation}
	
	Finally, by Proposition \ref{prop:diffeomorphism sigma_t}, we have $\psi(t,\sigma_t^{-1}(s)) = \lim_{x\to 0} \psi(t,X(t;s,x))$. Thus, using the dominated convergence theorem and equation~\eqref{eq:deriv-phi-x},
	\[ \int_s^T \psi(t,\sigma_t^{-1}(s))\, dt=\lim_{x\to 0} \int_s^T \psi(t,X(t;s,x)) \, dt=\vphi(t,0)\]
	for all $t\in (0,T)$. Replacing this last relation in Eq.~\eqref{eq:weak-etape2} we obtain  
	\begin{equation}\label{eq:weak-etape2-bis}
	\int_0^T \int_0^{x_c(t)} \psi(t,x) G( u(\sigma_t(x)) ) |\sigma_t(x)'|\, dx\, dt   = \int_0^T \vphi(t,0) G(t)\, dt\,.
	\end{equation}
	Inserting Eqs.~\eqref{eq:weak-etape1-bis} and \eqref{eq:weak-etape2-bis} into Eq.~\eqref{eq:intermediate-mild-1} ends the proof.
\end{proof}

Now we prove points (1) and (3) in Theorem \ref{th:compilation}. We can show that Eq.~\eqref{eq:ww} is satisfied by $f$ whenever $\vphi(t,x) = g(t)h(x)$, with $g\in \Cc^1_c(0,T)$ and $h\in\Cc_c^0([0,\infty))$ with $h'\in L^\infty(0,\infty)$. This follows from a standard regularization argument, together with the fact that $f$ belongs to $L^\infty((0,T);L^1(0,\infty))$,  Eq.~\eqref{eq:bound-1+x-f}, and the fact that  the rates are locally bounded.  Then, again by regularization, Eq.~\eqref{eq:ww} is shown to be true  for $h$ locally bounded and such that $h'\in L^\infty(0,\infty)$, namely,
\begin{multline*}
\int_0^T g'(t) \int_0^\infty h(x) f(t,x)\,dx  + \int_0^T g(t) \int_0^\infty v(t,x)
 h'(x) f(t,x)\,dx\, dt \\
+ h(0)\int_0^T g(t) G(t)\, dt = 0\,.
\end{multline*}
Here we used that $f$ belongs to  $L^\infty((0,T);L^1((0,\infty),(1+x)dx))$ and the sublinearity of $v$ in \eqref{A1}; 
 note that $h$ has a well-defined limit at the origin. This entails that the map  $t\mapsto \int_0^\infty h(x)f(t,x)\,dx$ has a bounded time derivative, which yields \eqref{eq:moment_formulation}. We have in particular that   $t\mapsto \int_0^\infty (1+x)h(x)f(t,x)\, dx$ is continuous for all $h\in\Cc_c^0(0,\infty)$, which is improved up to $h\in L^\infty(0,\infty)$ thanks to Lemma \ref{lem:compactness} and implies the claimed regularity of $f$. To finish the proof we analyze the limit in point (\ref{eq:trace}). Let $t\in(0,T)$, we have 
\[f(t,x) = G(\sigma_t(x)) |\sigma_t'(x)|\, \quad \mbox{a.e.}\ x\in(0,x_c(t))\,.\]
Since the right-hand side is continuous in $x$ we may choose a version of $f$ that is continuous on $(0,x_c(t))$. Then  from Eq. \eqref{eq:deriv stigmatx}
\[v(t,x)
f(t,x) =  G(\sigma_t(x))
\frac{w(t,x)}{w(\sigma_t(x),0)}
 e^{\left( \int_{\sigma_t(x)}^t (a\, \partial_x w)(\tau,X(\tau;t,x))
 \, d\tau \right)}\,.
 \]
Thanks to Proposition \ref{prop:diffeomorphism sigma_t}, the factor in front of the exponential converges to $G(t)$ as $x \to 0^+$. It remains to prove that 
\[\lim_{x\to 0^+} \int_{\sigma_t(x)}^t \left(a\,\partial_x w\right)(\tau,X(\tau;t,x))
\, d\tau = 0\,.
\]
Consider $x_0$ and $\delta$ given by\eqref{A6} and let $x<x_0$ so that for all $\tau\in(\sigma_t(x),t)$ we have $X(\tau;t,x)<x_0$. Then, by Lemma \ref{lem:decomposition a phi}, 
\[\int_{\sigma_t(x)}^t \left(a\, \partial_x w\right)(\tau,X(\tau;t,x))
\, d\tau \leq \frac 1 \delta \int_0^t\int_0^{x} |\partial_y(r,y)| \, dydr\,.\]
This last term vanishes as $x \to 0$, which concludes the proof.

Finally, uniqueness follows from a classical duality argument. Let $\psi \in \mathcal C^1_c(\Omega_T^*)$. We have that $\vphi(t,x)=-\int_t^{T} \psi(s,X(s;t,x))ds$ is a solution of 
\[\partial_t \varphi + v \partial_x \vphi = -\psi, \qquad \vphi(T,x)=0\,,\]
so that for any two solutions $f_1$ and $f_2$ with initial data $f^{\rm in}$, there holds that  
\[ \int_0^T \int_0^\infty \psi (f_1-f_2) =  0\,.\]

	\section{Annex: Proof of Lemma \ref{lem:newlem_E} for uniqueness}
	\label{Ann:6}
	
	We start by proving the following classical result on tail density with the notation of Sec. \ref{sec:uniqueness}.
	\begin{lemma} \label{lem:distribution formulation F}
		We have, for $i=1,\,2$, that $F_i \in L^\infty((0,T);L^1(0,\infty))\cap L^\infty(\Omega_T^*)$,  that $\partial_x F_i=-f_i$ belongs to $L^\infty((0,T);L^1((0,\infty),(1+x)\,dx))$ and also that $\partial_t F_i$ belongs to $ L^\infty((0,T);L^1(0,\infty))$. Moreover, they satisfy
		\begin{equation}\label{eq:mass_Fi}
		\int_0^\infty F_i(t,x) \, dx = \int_0^\infty xf_i(t,x)\, dx
		\end{equation}
		for all $t\in(0,T)$, and
		\begin{equation} \label{eq:weak equation on F+}
		\partial_t F_i + v_i \partial_x F_i = 0\,, \quad \text{in}\ \mathcal{D}'(\Omega_T^*)\,.
		\end{equation}
	\end{lemma}
	
	\begin{proof}
		Recall that $f_i$ belongs to $L^\infty((0,T);L^1((0,\infty),(1+x)\, dx))$. The boundedness of $F_i$ is an obvious consequence of the integrability of $f_i$ and the definition in Eq.~\eqref{eq:definition of F}, together with the regularity of the derivative in $x$. Integrability of $F_i$ and formula \eqref{eq:mass_Fi} follow from Tonelli's Theorem. Eq. \eqref{eq:weak equation on F+} is obtained using test functions of the form $\vphi(t,x)=\int_0^x \psi(t,y)\, dy$, for $\psi\in \Dc(\Omega_T^*)$, in Eq.~\eqref{eq:LS-weak-time-density-boundary} together with Fubini's theorem. Finally, the regularity of the time derivatives follows from Eq.~\eqref{eq:weak equation on F+}, the sublinearity of the rates and the regularity of $\partial_x F_i=f_i$. 
	\end{proof}
	
	\begin{proof}[Proof of Lemma \ref{lem:newlem_E}]
	By Lemma \ref{lem:distribution formulation F} we deduce
	\begin{equation} \label{eq:weak equation on E+}
	\partial_t E = -v_1 \partial_x F_1 + v_2 \partial_x F_2 = - v_1\partial_x E + aw f_2 \, , \quad \text{in}\, \mathcal{D}'(\Omega_T^*)\,.
	\end{equation}
 By Lemma \ref{lem:distribution formulation F} and Eq.~\eqref{eq:weak equation on E+}, for any real function $\beta$ defined on $\Rb$, continuously differentiable with bounded derivatives, we have
	\begin{equation}
	\nonumber
	\partial_t \beta(E) = - v_1\partial_x \beta(E) + aw  f_2 \beta'(E) \, , \quad \text{in}\ \mathcal{D}'(\Omega_T^*)\,.
	\end{equation}
	In particular we are led to 
	\begin{multline}
	\nonumber
	\frac{d}{dt}\int_0^\infty \vphi(x)\beta(E(t,x))\, dx = \int_0^\infty \partial_x [v_1(t,x)\vphi(x)] \beta(E(t,x)) \, dx \\
	+ \int_0^\infty a(x)w(t)  f_2(t,x) \beta'(E(t,x))\vphi(x) \, dx
	\end{multline}
	for all $\vphi$ belonging to $\mathcal D(0,\infty)$. Note that the distributional derivative $\partial_t \beta(E)$
	belongs to $L^\infty(0,T)$. This is due to $E$ being bounded, $f_2$ being integrable against $(1+x)$, the sublinearity of $a$ in \eqref{eq:sublinearity}, the fact that $u_1$ and $u_2$ are bounded and the boundedness of $\beta'$. We obtain
	\begin{multline}\label{eq:pregronwall-Ep-regular_intermediate}
	\int_0^\infty \vphi(x)\beta(E(t,x))\, dx  \leq   \int_0^\infty \vphi(x)\beta(E(0,x))\, dx \\
	+ \int_0^t \int_0^\infty \partial_x[v_1(s,x)\vphi(x)]  \beta(E(s,x)) \, dx\, dt \\
	+ \|\beta'\|_{L^\infty} \int_0^t |w(s)|  \int_0^\infty a(x) |\vphi(x)| f_2(s,x) \, dx \, dt 
	\end{multline}
	for any $\vphi$ belonging to $\Dc(0,\infty)$. 
	
	To obtain Eq.~\eqref{eq:pregronwall-Ep-regular} in Lemma \ref{lem:newlem_E}, we use a regularization procedure. Let $\vphi$ a nonnegative function belonging to $\mathcal C^0([0,\infty))$, such that $\vphi$ vanishes in a neighborhood of zero, and $\vphi'\in L^\infty(0,\infty)$ is compactly supported. To be able to substitute $\vphi$ into Eq.~\eqref{eq:pregronwall-Ep-regular_intermediate}, we need to regularize it to make it infinitely derivable and to truncate its support for large $x$. For each $R>1$, denote by $\chi_R$ a real function in $\Dc(\Rb)$ with $0\leq \chi_R\leq 1$, such that $\chi_R=1$ on $(0,R)$, with compact support in $(0,R+1)$, and $|\chi_R'|\leq 2$ on $(R,R+1)$. Let $\{g^\veps\}$ be a standard mollifying sequence. Define $\vphi_R^\veps=\vphi_R*g^\veps$  with $\vphi_R=\vphi \chi_R$ on $(0,\infty)$. We shall substitute $\vphi_R^\veps$ into \eqref{eq:pregronwall-Ep-regular_intermediate} and take the limits $\epsilon \to 0$ and $R\to \infty$ in turn.  Note that $\vphi_R^\veps$ converges uniformly to $\vphi_R$ on $\Rb$ as $\veps\to  0$. Moreover, the support of $\vphi_R^\veps$ is contained in $[0,R+1+\veps]$. For the time being, assume that $\beta$ is a nonnegative function on $\Rb$, continuously differentiable with $|\beta'|\le 1$ and $\beta(0)=0$. Observe that $\beta(y)\leq |y|$ for all $x\in \Rb$; since $|E(t,x)|$ is bounded on $\Omega_T$, it follows that
	\begin{equation*}
	\lim_{\veps \to 0}\int_0^\infty \vphi_R^\veps(x)\beta(E(t,x))  \, dx \\
	= \int_0^\infty \vphi_R(x)\beta(E(t,x)) \, dx <\infty\,,
	\end{equation*}
	for any $t\in[0,T)$. Then, since $f_2$ belongs to $L^\infty\left((0,T);L^1((0,\infty);(1+x)\, dx)\right)$ and $a$ is sublinear by \eqref{eq:sublinearity}, we have 
	\begin{equation*}
	\lim_{\veps \to 0} \int_0^t \int_0^\infty a(x)  f_2(t,x) \vphi^\veps_R(x) \, dx \, dt \\
	=  \int_0^t \int_0^\infty a(x)  f_2 (t,x) \vphi_R(x) \, dx \, dt
	\end{equation*}
	for all $t\in(0,T)$. Now, we remark that
	\begin{multline*}
	\int_0^t \int_0^\infty \partial_x(v_1(s,x) \vphi^\veps_R(x))  \beta(E(s,x)) \, dx\, dt \\
	= \int_0^t \int_0^\infty \left\{\partial_xv_1(s,x) \vphi^\veps_R(x) + v_1(s,x) \vphi_R'*g^\veps(x) \right\}  \beta(E(s,x)) \, dx\, dt.
	\end{multline*}
	On one hand, as $a$ and $b$ are continuously differentiable on $(0,\infty)$, $\vphi_R$ is compactly supported and $E$ belongs to $L^\infty(\Omega_T)$, we have  
	\begin{multline*}
	\lim_{\veps \to 0} \int_0^t \int_0^\infty \partial_xv_1(s,x) \vphi^\veps_R(x)   \beta(E(s,x)) \, dx\, dt\\
	=  \int_0^t \int_0^\infty \partial_x v_1(s,x) \vphi_R(x)   \beta(E(s,x)) \, dx\, dt. 
	\end{multline*}
	On the other hand, note that $(\vphi\chi_R)'$ is bounded with compact support, thus $(\vphi\chi_R)'*g^\veps$ converges to $(\vphi\chi_R)'$ almost everywhere. But $\vphi\chi_R$ has compact support and $a$ and $b$ are continuous, hence  bounded on this support. Moreover, $E$ belongs to $L^\infty(\Omega_T)$, so, via the dominated convergence theorem we have
	\begin{multline*}
	\lim_{\veps \to 0} \int_0^t \int_0^\infty v_1(s,x) \vphi_R'(x)*g^\veps(x)   \beta(E(s,x)) \, dx\, dt \\ =  \int_0^t \int_0^\infty v_1(s,x) \vphi_R'(x)   \beta(E(s,x)) \, dx\, dt .
	\end{multline*}
	Recapitulating, using that $\vphi_R \leq \vphi$, we get
	\begin{multline} 
	\label{eq:inter E+ vphiRannex}
	\int_0^\infty \!\!  \vphi_R(x)\beta(E(t,x))\, dx  \leq   \int_0^\infty \! \! \vphi(x) |E(0,x)|\, dx \\
	+\int_0^\infty \! \! \partial_x[v_1(t,x)\vphi_R(x)]  \beta(E(t,x))\, dx
	\\
	+ \int_0^t |w(s)|\int_0^\infty a(x)  f_2 (s,x) \vphi(x) \, dx \, ds \,.
	\end{multline}
	We may now pass in the limit $R \to \infty$ in Eq.~\eqref{eq:inter E+ vphiRannex}. We have, for any $t\in(0,T)$, by integration by parts (recall that $\vphi$ vanishes around the origin),
		\begin{equation}\label{eq:IPP1}
		\int_{0}^\infty  \partial_x[v_1(t,x)\vphi_R(x)]  \beta(E(t,x))\, dx = - \int_{0}^\infty  v_1(t,x) \vphi_R(x) \partial_x E(t,x) \beta'(E(t,x))\, dx 
		\end{equation}
		 As $\beta'$ is bounded and $ v_1(t,x)\partial_x E(t,x) $ integrable, we may pass in the limit $R\to\infty$ in the right-hand side of Eq.~\eqref{eq:IPP1} using the dominated convergence theorem. As $\vphi'$ is compactly supported and $\vphi$ vanishes around the origin, $\partial_x[v_1(t,x)\vphi(x)]$ is bounded, and we may perform an integration by parts in the other way, to obtain
		 \[\lim_{R\to\infty}  \int_{0}^\infty \! \! \partial_x[v_1(t,x)\vphi_R(x)]  \beta(E(t,x))\, dx  = \int_{0}^\infty  \partial_x[v_1(t,x)\vphi(x)] \beta(E(t,x))\,.\]
		 Thus,  letting $R\to \infty$ in \eqref{eq:inter E+ vphiRannex}, we get
			\begin{multline} 
		\label{eq:inter E+ vphiR2}
		\int_0^\infty \!\!  \vphi(x)\beta(E(t,x))\, dx  \leq   \int_0^\infty \! \! \vphi(x) |E(0,x)|\, dx \\
		+\int_0^\infty \! \! \partial_x[v_1(t,x)\vphi(x)]  \beta(E(t,x))\, dx
		\\
		+ \int_0^t |w(s)|\int_0^\infty a(x)  f_2 (s,x) \vphi(x) \, dx \, ds \,.
		\end{multline} 
	
	We then use the approximation of the absolute value	$\beta(x)=|x|-\epsilon/2$ for $|x|>\veps$ and $\beta(x)=\tfrac 1 {2\veps}x^2$ for $|x|\le \veps$ in the above equation \eqref{eq:inter E+ vphiR2} and we let $\veps \to 0$ (note again that $\partial_x[v_1(t,x)\vphi(x)]$ is bounded). We thus obtain Eq.~\eqref{eq:pregronwall-Ep-regular}.
	
We now prove Eqs.~\eqref{eq:inter-w-1}-\eqref{eq:control_on_E+(0)}. For $t\in(0,T)$,
	\[|w(t)| = \left| \int_0^\infty x f_1(t,x)\, dx - \int_0^\infty x f_2(t,x)\, dx \right|\,.\]
	In virtue of \eqref{eq:mass_Fi}, $E$ belongs to $L^\infty((0,T);L^1(0,\infty))$ and
	\begin{equation*}
	|w(t)| \leq \int_0^\infty |E(t,x)|\, dx
	\end{equation*}
	for all $t\in(0,T)$. Finally, thanks to Lemma \ref{lem:moments equations},
	\[F_i(t,0) = F_i(0,0) + \int_0^t \mathfrak n(u_i(s))\, dt\]
	for $i=1$, $2$ and hence by \eqref{H5'} there exists $K_\mathfrak{n}$ such that
	\begin{equation*}
	|E(t,0)| \leq |E(0,0)| + K_{\mathfrak n} \int_0^t |w(s)| \, ds
	\end{equation*}
	where $K_{\mathfrak n}$ is the Lipschitz constant of $\mathfrak n$ on $[\Phi_0,\rho]$.
		\end{proof} 
	
\section*{Acknowledgments} 

The authors would like to thank Boris Andreianov and Guy Barles (Institut Denis Poisson, Universit\'e de Tours) for many interesting and helpful discussions on the subject.
We also warmly thank the reviewers for their careful reading of the manuscript, and their valuable remarks that help us improve its quality.

\noindent J. C. acknowledges support from MICINN, projects MTM2017-91054-EXP and RTI2018-098850-B-IOO; he also acknowledges support from Plan Propio de Investigaci\'on, Universidad de Granada, Programa 9 -partially through FEDER (ERDF) funds-. E. H. acknowledges support from  FONDECYT Iniciaci\'on n$^\circ$ 11170655. R. Y. does not have to thank the French National Research Agency for its financial support but he kindly thanks it for the excellent reviews embellished with arguments based on scientific and cultural novelties in the expertise of his yearly application file during the last four years.

Part of this work was done while J. C. and R. Y. were visiting the Departamento de Matem\'atica at  Universidad del B\'io-B\'io and while E. H. and J. C. were visiting Institut Denis Poisson at Universit\'e de Tours and INRAE Nouzilly. J.C. thanks Universit\'e de Tours for a visiting position during last winter.

\end{document}